\documentclass[review]{elsarticle}

\usepackage{lineno,hyperref}
\modulolinenumbers[5]

\journal{}

\usepackage{amsmath}
\usepackage{graphics}
\usepackage{epsfig}
\usepackage{algorithm}
\usepackage{verbatim}
\usepackage{stmaryrd}
\usepackage{epstopdf}
\usepackage{xcolor}
\usepackage{color}
\usepackage{bbm}
\usepackage{amssymb}
\usepackage{mathrsfs}
\usepackage{caption}
\usepackage{amsthm}
\usepackage{nomencl}
\usepackage{multirow}
\usepackage{chngpage}
\usepackage{array}
\usepackage{comment}
\usepackage{enumerate}  
\usepackage{graphicx}
\usepackage{transparent}
\usepackage{mathtools}
\usepackage{tikz}
 \usepackage{algpseudocode}
 \usepackage{arydshln}
\usetikzlibrary{shapes.geometric, arrows}
\usetikzlibrary{matrix}
\biboptions{numbers,sort&compress}
\usepackage[top=1in, bottom=1in, left=1.25in, right=1.25in]{geometry}
\makenomenclature
\allowdisplaybreaks

\theoremstyle{plain}
\newtheorem{thm}{\bf Theorem}[section]

\newtheorem{pro}[thm]{\bf Proposition}

\newtheorem{claim}[thm]{\bf Claim}

\theoremstyle{definition}
\newtheorem{rem}[thm]{\bf Remark}
\newtheorem{defn}[thm]{\bf Definition}
\newtheorem{exa}[thm]{\bf Example}

\newcommand{\red}[1]{\textcolor{black}{#1}}
\newcommand{\blue}[1]{\textcolor{black}{#1}}
\newcommand{\mage}[1]{\textcolor{black}{#1}}
\newcommand{\cyan}[1]{\textcolor{black}{#1}}
\newcommand\rk{{\rm rank\,}} 
\newcommand\im{{\rm Im\,}} 
\newcommand\rD{{\rm D}}

\algnewcommand\Assum{\item[\textbf{Assumption 2:}]}
\algnewcommand\Stepz{\item[\textbf{Step $0$:}]}
\algnewcommand\Stepk{\item[\textbf{Step $k$:}]}
\algnewcommand\Rep{\item[\textbf{Repeat:}]}
\algnewcommand\Result{\item[\textbf{Result:}]}
\begin{document}
\begin{frontmatter}

\title{Geometric analysis  of nonlinear differential-algebraic equations via nonlinear control theory}

\author[First]{Yahao Chen} 
\author[Second]{Witold Respondek} 

\address[First]{Bernoulli Institute, University of Groningen, The Netherlands (e-mail: yahao.chen@rug.nl); }
\address[Second]{Normandie Universit\'{e}, INSA-Rouen, LMI, France (e-mail: witold.respondek@insa-rouen.fr).}

\begin{abstract}
For nonlinear differential-algebraic equations (DAEs), we define two kinds of equivalences, namely, the external and internal equivalence. Roughly speaking, the word ``external'' means  that we consider a DAE (locally) everywhere and ``internal'' means that  we consider the DAE on its (locally) maximal invariant submanifold (i.e., where its solutions exist) only. First, we revise the geometric reduction method in DAEs solution theory   and formulate an implementable algorithm to realize that method. Then a procedure called explicitation with driving variables is proposed to connect nonlinear DAEs with nonlinear control systems and we show that the driving variables of an explicitation system can be reduced under some involutivity conditions. Finally, due to the explicitation, we  use some notions from nonlinear control theory to derive two  nonlinear generalizations of the Weierstrass form.
\end{abstract}
\end{frontmatter}
\section{Introduction}\label{Chap3sec1}
Consider a nonlinear differential-algebraic equation (DAE) of {the} form
\begin{align}\label{Eq:DAE0}
\Xi:	E(x)\dot x= F(x),
\end{align}
where $x\in X$ is a vector of the generalized states {and $X$ is} an open subset of $\mathbb R^n$ (or an $n$-dimensional manifold).
\vspace{-0.5cm}
\begin{center}
\begin{tikzpicture}
\matrix (m) [matrix of math nodes,row sep=0.5em,column sep=3em,minimum width=1em]
{TX &   \\
	&  \mathbb R^{l} \\
X&\\};
\path[-stealth]
(m-1-1) edge node [left, xshift=-0.1cm] {$\pi$} (m-3-1)
 edge node [above, yshift=0.1cm ] {$E$} (m-2-2)
(m-3-1) edge   node [below,yshift=-0.1cm ] {$F$} (m-2-2);
\end{tikzpicture}
\end{center}
\vspace{-0.5cm}
 The maps $E:TX\rightarrow\mathbb{R}^{l}$ and $F: X\rightarrow{\mathbb{R}^l}$  (see the above diagram, where  $ \pi :TX\rightarrow X$ is the canonical projection)  are smooth and the word ``smooth'' will mean throughout this paper $\mathcal{C}^{\infty}$-smooth . We will denote a DAE of \red{the} form (\ref{Eq:DAE0}) by $\Xi_{l,n}=(E,F)$ or, simply, $\Xi$. 
Equation (\ref{Eq:DAE0}) is affine with respect to the velocity $\dot x$, so sometimes it is called a quasi-linear DAE (see e.g., \cite{rabier2002theoretical,riaza2008differential}) and can be considered as an affine Pfaffian system {since the rows $E^i$ of $E$ are actually differential $1$-forms on $X$} (for linear Pfaffian systems, see e.g. \cite{kobayashi1963foundations}), \red{so $E$ is a $\mathbb R^l$-valued differential 1-form on $X$}.
 A  semi-explicit   DAE is of the form
\begin{align}\label{Eq:DAE1}
\Xi^{SE}:	\left\lbrace {\begin{array}{c@{}l}
	\dot x_1&=F_1(x_1,x_2)\\
	0&=	F_2(x_1,x_2),
	\end{array}} \right.
\end{align}
where $x_1\in X_1$   is a vector of state variables and $x_2\in X_2$ is a vector of algebraic \red{or free} variables (since there are no differential equations for $x_2$) with $X_1$ and $X_2$ being open subsets of $\mathbb R^q$ and $\mathbb R^{n-q}$, respectively (or  $q$- and $(n-q)$-dimensional manifolds, respectively), the maps $F_1:X_1\times X_2\rightarrow TX_1$ and $F_2:X_1\times X_2\rightarrow \mathbb R^{l-q}$ are smooth.   A linear DAE of the form
\begin{align}\label{Eq:DAE2}
\Delta: E\dot x=Hx
\end{align}
will be denoted by $\Delta_{l,n}=(E,H)$ or, simply, $\Delta$, where $E\in \mathbb{R}^{l\times n}$ and $H\in \mathbb{R}^{l\times n}$. Both the semi-explicit DAE $\Xi^{SE}$ and the linear DAE $\Delta$ can be seen as special {cases} of  DAE $\Xi$. 
The motivation of studying DAEs is their frequent {presence} in {modelling} of  practical systems as electrical circuits \cite{riaza2008differential,riaza2013daes}, chemical processes \cite{byrne1988differential,pantelides1988mathematical}, mechanical systems \cite{rabier2000nonholonomic,betsch2002dae,blajer2007control}, etc.

There are  three main results of this paper. {The first result concerns analyzing a DAE} (locally) everywhere (i.e., externally) {or} considering the restriction of the DAE {to a submanifold} (i.e., internally), which corresponds {to} the external equivalence (see Definition \ref{Def:ex-equivalence}) and the internal equivalence (see Definition \ref{Def:in-equivalence}), respectively. The difference {between} the two {equivalences} will be illustrated by their relations with the solutions. In order to analyze the existence of solutions, we use a concept called  \emph{locally maximal invariant submanifold} (see Definition \ref{Def:invariant manifold}),  which is a submanifold  where the solutions of a DAE  exist and can be constructed via a geometric reduction method shown in Section~\ref{section:2}. Note that the geometric reduction method is not new in the  theory of nonlinear DAEs, see e.g.,   \cite{reich1990geometrical,reich1991existence,rabier1994geometric,rabier2002theoretical,riaza2008differential} and the recent papers \cite{chenMTNS,berger2016controlled,Berger2016zero}. In the present paper, we will show a practical implementation of this method via an algorithm \red{summarized in} Appendix.  Note that considering only the restriction of a DAE means {that} we only care about {where and} how the solutions of {that} DAE evolve. However, when \red{a} nominal point is not on the maximal invariant submanifold (which is common for practical systems, since an initial point could be anywhere), there are no solutions passing through the point but we still want to steer the solutions to the submanifold and \red{thus we} must follow the rules indicated by the ``external'' form of the DAE, thus considering  {DAEs} everywhere is also important, see our recent publication \cite{chen2021ADHS}, where we use external equivalence to study jump solutions of nonlinear DAEs.

The second result of this paper is a nonlinear counterpart of the results of \cite{chen2019a}, in which  we have shown that one can associate a class {of} linear control systems  to any linear DAE (by the procedure of explicitation for linear DAEs). In the present paper, to {any nonlinear} DAE, by introducing extra variables (called driving variables), we can attach a class of nonlinear control systems. Moreover, we show that the driving variables in this explicitation procedure can be \red{fully} reduced under some involutivity conditions {which explains} when a DAE $\Xi$ is ex-equivalent to a semi-explicit DAE $\Xi^{SE}$.

It is well-known (see e.g., \cite{Kron90},\cite{Gant59b}) {that} any linear DAE $\Delta$ of the form (\ref{Eq:DAE2}) is ex-equivalent (via linear transformations) to the Kronecker canonical form \textbf{KCF}. In particular, if $\Delta$ is \emph{regular}, i.e., the matrices $E$ and $H$ are square ($l=n$) and $\left| {sE - H} \right|\not\equiv0$, $\forall s\in \mathbb{C} $, then $\Delta$ is ex-equivalent (also via linear transformations) to the Weierstrass form \textbf{WF} \cite{Weie68}  (see (\ref{Eq:WF}) below). The studies on normal forms and canonical forms of DAEs can be found in \cite{Weie68,Kron90,loiseau1991feedback,BERGER20124052,Berger2012} for the linear  case and in \cite{rouchon1992kronecker,KumaDaou99,Berger2016zero} for the nonlinear case.  The last result of this paper is to use {such} concepts as zero dynamics, \cyan{relative degree and invariant distributions} of the nonlinear control theory \cite{Isid95,nijmeijer1990nonlinear} to derive   nonlinear generalizations of the \textbf{WF}. In the linear case, canonical forms as the \textbf{KCF} and the \textbf{WF} are closely related to a geometric concept named the Wong sequences \cite{Wong74} (see Remark \ref{Rem:wong} below). \cyan{In} \cite{BERGER20124052}, relations between the \textbf{WF} and the Wong sequences  have been built and in \cite{Berger2012}, the importance of the Wong sequences for the geometric analysis of linear DAEs is reconfirmed. \cyan{In the present paper, we propose generalizations of the Wong sequences for nonlinear DAEs and show their importance in analyzing structure properties.} 

 This paper is organized as follows. In Section \ref{section:2}, we discuss \red{the} existence of solutions of DAEs by revising the geometric reduction method. In Section \ref{subsection:3.1}, we compare the notions of external equivalence and internal equivalence and discuss the uniqueness of DAEs solutions via the notion of internal regularity. In Section \ref{subsection:3.2}, we propose the explicitation (with driving variables) procedure to connect nonlinear DAEs to nonlinear control systems.  In Section \ref{subsection:3.3}, we show when a nonlinear DAE is externally equivalent to {a} semi-explicit one and how this problem {is} related to the explicitation procedure.   \cyan{Two nonlinear generalizations of the Weierstrass form are given} in Section \ref{subsection:3.4}. Finally, Section~\ref{sec:Proofs} and Section~\ref{section:4} contain  proofs and the conclusions, respectively. In Appendix of Section \ref{sec:Appendix}, we show a  recursive algorithm which implements  the geometric reduction method.
 
  The following notations will be used throughout the paper.  We use 
${\mathbb{R}^{n \times m}}$ to denote the set of real valued matrices with $n$ rows and $m$ columns, $GL\left( {n,\mathbb{R}} \right)$ to denote  the group of nonsingular matrices of $\mathbb{R}^{n \times n}$ and   $I_n$ to denote the $n\times n$-identity  matrix.  For a linear map $L$, we denote by $\rk L$, $\ker L$ and $\im L$, the rank, the kernel and the image of $L$, respectively. Denote by
$T_xM$  the tangent space of a submanifold $M$ of $\mathbb{R}^n$ at $x\in M$ and \red{by} $\mathcal C^k$  the class of functions which are $k$-times differentiable.  For a smooth map $f:X\to \mathbb R$, we  denote its differential by ${\rm d} f=\sum^n_{i=1}\frac{\partial f}{\partial x_i}{\rm d}x_i=[\frac{\partial f}{\partial x_1},\ldots,\frac{\partial f}{\partial x_n}]$  and for a vector-valued map $f:X\to\mathbb R^m$, where $f=[f_1,\ldots, f_m]^T$, we denote its differential by ${\rm D} f =\left[ \begin{smallmatrix}
{\rm d} f_1\\
\vdots\\
{\rm d} f_m
\end{smallmatrix}\right]$.  For two column vectors  $v_1\in \mathbb R^m$ and $ v_2\in \mathbb R^n$, we  write $(v_1,v_2)=[v^T_1,v^T_2]^T\in \mathbb R^{m+n}$. 
\section{The geometric reduction method revisited}\label{section:2}
In this section, we revise   the geometric reduction method in \red{the} DAEs solution theory, \red{other} formulations of this method \red{can} be consulted in   Section 3.4 of \cite{riaza2008differential}, Chapter IV of \cite{rabier2002theoretical} and  \cite{chenMTNS} for DAEs and \cite{berger2016controlled} for DAE control systems.  We start from the definition of a solution for a DAE.
\begin{defn}
	A solution of a DAE $\Xi_{l,n}=(E,F)$ is a $\mathcal C^1$-curve $x:I\rightarrow {X}$ defined on an open interval $I$ such that for all $t\in I$, the curve $x(\cdot)$ satisfies $E\left( {x(t)} \right)\dot x(t) = F\left( {x(t)} \right)$.
\end{defn}
Throughout this paper, {we will be interested only in solutions of $\Xi$ that are at least $\mathcal C^1$.}    A given point $x_0$ is called \emph{consistent} (\red{or \emph{admissible}}) if there exists at least one solution $x(\cdot)$ of $\Xi$ satisfying $x(t_0)=x_0$ (i.e., $E(x_0)\dot x(t_0)=F(x_0)$) for a certain $t_0\in I$, we will denote  by $S_c$ \emph{the consistency set}, i.e., the set of all consistent points. 
\begin{defn}[invariant and locally invariant submanifolds]\label{Def:invariant manifold}
	Consider a DAE $\Xi_{l,n}=(E,F)$  defined on $X$. A smooth connected embedded submanifold  $M$ of $X$ is called \emph{invariant} if  for any point $x_0\in M$, there exists a solution  $x:I\rightarrow X$ of $\Xi$ such that $x(t_0)=x_0$ \mage{for a certain $t_0\in I$} and  $x(t)\in M$ for all $t\in I$. Given a  point $x_p\in X$, we will say that a submanifold $M$ \cyan{containing $x_p$} is  \emph{locally invariant} (around $x_p$) if  there exists an open neighborhood $U\subseteq X$ of $x_p$ such that $M\cap U$ is invariant. 
\end{defn} 
\begin{pro}\label{Pro:invariant submanifold}
	Consider a DAE $ \Xi_{l,n}=(E,F)$ and fix a point $x_p$. Let $M$ be a smooth connected embedded submanifold containing $x_p$. If $M$ is a locally invariant submanifold around $x_p$, then $F(x)\in E(x)T_xM$  for all $x\in M$ around $x_p$.  Conversely, \red{assume that there exists an open neighborhood $U$ of $x_p$ such that, at all $x\in M\cap U$, we have $F(x)\in E(x)T_xM$ and, additionally, $\dim E(x)T_xM=const.$,}  then $M$ is a locally invariant submanifold.
\end{pro}
The proof is given in Section \ref{sec:prf_prop}.
\begin{rem}
 Note that the  assumption  that $\dim E(x)T_xM=const.$ of Proposition \ref{Pro:invariant submanifold} is not a necessary condition to conclude that $M$ is an invariant submanifold,  but it   excludes singular points of DAEs and helps to view a DAE as an ordinary differential equation (ODE) defined on the invariant submanifold. 
Take the following DAE  for an example:$$ \Xi_{1,1}:x\dot x=x^2,$$
	where $x\in X=\mathbb R$. \red{Let $M=X$, then clearly, $F(x)=x^2\in x\cdot T_xX$, at any $x\in M=\mathbb R$. }We have  $\dim\, E(x)T_{x}M $ equals $1$ for $x\neq0$ and is $0$ for $x=0$, so $\dim\, E(x)T_{x}M\neq const.$, for all $x\in M$. \red{Nevertheless, for any $x_0\in M=\mathbb R$, there exists a unique solution $x(t)$ satisfying $x(0)=x_0$, namely, $x(t)=e^tx_0$. Therefore $M=\mathbb R$ is an invariant submanifold.}
\end{rem}
A locally invariant submanifold $M^*$ (around $x_p$) is called \mage{locally} \emph{maximal}, if there exists a neighborhood $U$ of $x_p$ such that for any other locally invariant submanifold $M$, we have $M\cap U \subseteq M^*\cap U$. The   \emph{geometric reduction method} for DAEs is the following recursive procedure which   can be used to construct   locally maximal invariant submanifold $M^*$.
\begin{defn}[geometric reduction method]\label{Def:grm}
Consider a DAE $\Xi_{l,n}=(E,F)$,  fix a point $x_p\in X$ and let $U_0$ be an open connected subset of $X$ containing $x_p$.      Set $M_0=X$, $M^c_0=U_0$.   Suppose that  there exist  an open neighborhood $U_{k-1}$ of $x_p$ and a sequence of smooth connected embedded submanifolds $M^c_{k-1}\subsetneq\dots \subsetneq M^c_0$ of $U_{k-1}$ for a certain $k\ge 1$, has been constructed. Define recursively
\begin{align}\label{Eq:Mseq}
{M_{k}} := \left\{ {x \in M^c_{k-1}:F(x) \in E(x){T_x M^c_{k-1}}}\right\}.
\end{align} 
{Then either} $x_p\notin M_k$ or $x_p\in M_k$, and in the latter case, assume that there exists a neighborhood $U_k$ of $x_p$ such that $M^c_k=M_k\cap U_k$ is a smooth  embedded submanifold (which can always be assumed connected by taking $U_k$ sufficiently small).
\end{defn}
\begin{rem}\label{Rem:wong}
	For a linear DAE $\Delta=(E,H)$ of the form (\ref{Eq:DAE2}), define a sequence of subspaces (one of the Wong sequences \cite{Wong74}) by	
	$$ 
	\mathscr{V}_0 =\mathbb{R}^n, \ \ \mathscr{V}_{k}=H^{-1}E\mathscr{V}_{k-1}, \ \ k\ge 1.
	$$
	If we apply the iterative construction of $M_k$ by (\ref{Eq:Mseq}) to the  DAE $\Delta$, we get $M^c_k=\mathscr V_{k}$, $\forall k\ge 0$. Thus the sequence of submanifolds $M_k$ can be seen as a nonlinear generalization of the   sequence $\mathscr V_k$.
\end{rem}
 The following proposition shows that the geometric reduction method above can be used to construct locally maximal invariant submanifold $M^*$ and \red{to deduce} that the consistency set $S_c$, on which the solutions  exist, coincides locally with $M^*$.
\begin{pro} \label{Pro:invariant manifold}
In the geometric reduction method of Definition \ref{Def:grm}, there always exists $k^*\le n$ such that either $k^*$ is the smallest integer for which $x_p\notin M_{k^*+1}$  or $k^*$ is the smallest integer  such that $x_p\in M^c_{k^*+1}$ and $M^c_{k^*+1}\cap U_{k^*+1}=M^c_{k^*}\cap U_{k^*+1}$. In the latter case, we   assume that $\dim\, E(x)T_xM^c_{k^*+1}=const.$ in a neighborhood $U^*\subseteq U_{k^*+1}$ of $x_p$ {in $X$ and then}
\begin{itemize}
\item [(i)]  $x_p$ is  consistent   and $M^*=M^c_{k^*+1}$ is a locally maximal invariant submanifold around $x_p$.
\item [(ii)] $M^*$ coincides locally with the consistency set $S_c$, i.e., $M^*\cap U=S_c\cap U^*$ (take a smaller $U^*$ if necessary).
\end{itemize}
\end{pro}
The proof is given in Section \ref{sec:prf_prop}.  Note that the geometric method can be implemented in \red{practice} via an algorithm which we propose in Appendix of the present paper \red{and} the results of Proposition~\ref{Pro:ismsol} and Theorem \ref{Thm:NWF} below will be based on that algorithm. 
 \section{External equivalence, internal equivalence and  internal regularity}\label{subsection:3.1}
Two linear DAEs $E\dot x=Hx$ and $\tilde E \dot {\tilde x}=\tilde H \tilde x$ are called  {externally equivalent \cite{chen2019a} or strictly equivalent \cite{Gant59b}}, if there exist constant invertible matrices $Q$ and $P$ such that $QEP^{-1}=\tilde E$ and $QHP^{-1}=\tilde H$. Analogously, we define the external equivalence of two nonlinear DAEs as follows. 
\begin{defn}[external equivalence]\label{Def:ex-equivalence}
	Two DAEs $\Xi_{l,n}=(E,F)$ and $\tilde \Xi_{l,n}=(\tilde E,\tilde F)$ defined on $X$ and $\tilde X$, respectively, are called externally equivalent, shortly ex-equivalent, if there exist a diffeomorphism $\psi: X\rightarrow \tilde X$ and    {$Q: X\rightarrow GL(l,\mathbb{R})$} such that
	$$
	\psi^*\tilde E=QE  \text{\ \ \ and \ \ \ } \psi^*\tilde F=QF,
	$$
\cyan{where $\psi^*\tilde E$ and $\psi^*\tilde F$ denote the pull-back \cite{kobayashi1963foundations} of the $\mathbb R^l$-valued differential $1$-form $\tilde E$ on $\tilde X$ and $\mathbb R^l$-valued function  $\tilde F$ ($0$-form) on $\tilde X$, respectively}, that is,
	\begin{align}\label{Eq:ex-eqnldae}
	\tilde E(\psi (x))=Q(x)E(x)\left( \frac{\partial \psi (x)}{\partial x}\right) ^{-1} \ \ {\rm and} \ \
 	\tilde F(\psi (x))=Q(x)F(x) .
	\end{align}
	The ex-equivalence of two DAEs {will be} denoted by $\Xi\mathop  \sim \limits^{ex}\tilde \Xi$. If  $\psi: U\rightarrow \tilde U$ is a local diffeomorphism between neighborhoods $U$ of $x_p$ and $\tilde U$ of $\tilde x_p$, and $Q(x)$ is defined on $U$,  we will speak about local ex-equivalence.
\end{defn}
The following observation \red{relates ex-equivalence with solutions.}
\begin{rem}\label{Obs:1}
The ex-equivalence preserves trajectories, i.e., for two DAEs $\Xi\mathop  \sim \limits^{ex}\tilde \Xi$, if a  $\mathcal{C}^{1}$-curve  $x(\cdot)$  is a solution of $\Xi$ passing through $x_0=x(t_0)$, then $\tilde x=\psi  \circ x $ is a solution of $ \tilde \Xi $ passing through $\tilde x_0=\psi(x_0)$; but even if  we can smoothly conjugate all trajectories of  two DAEs, they are not necessarily ex-equivalent. For example, consider
$\Xi_1=(E_1,F_1)$ and $ \Xi_2=(E_2,F_2)$, where 
$
E_1(x) = \left[ {\begin{smallmatrix}
	0&0&1\\
	0&0&0\\
	0&0&0
	\end{smallmatrix}} \right]$,   $F_1(x) = \left[  {\begin{smallmatrix}
	x_3^2\\
	x_1\\
	x_2
	\end{smallmatrix}} \right]$, $ E_2(x)=\left[ {\begin{smallmatrix}
0&x_1&1\\
0&0&0\\
0&0&0
	\end{smallmatrix}} \right]$,   $F_2(x) = \left[  {\begin{smallmatrix}
	x_3^2\\
	x_1\\
	x_2
	\end{smallmatrix}} \right] $. \red{Then for both DAEs $\Xi_1$ and $\Xi_2$, the maximal invariant submanifold is $M^*=\left\lbrace (x_1,x_2,x_3)\in \mathbb R^3\,|\, x_2=x_3=0\right\rbrace $ and for any $(x_{10},x_{20},x_{30})=(0,0,x_{30})\in M^*$, the unique solution of both systems is $x_1(t)=x_2(t)=0$, $x_3(t)=\frac{x_{30}}{1-x_{30}t}$. Nevertheless, the DAEs are not ex-equivalent since the distribution $\ker E_1$ is involutive but the distribution $\ker E_2$ is not (clearly, the ex-equivalence of two DAEs preserves the involutivity of $\ker E_1$ and $\ker E_2$ since if $\Xi_1\overset{ex}{\sim}\Xi_2$, via $Q$ and $\psi$, then $\ker E_2=\frac{\partial \psi}{\partial x}\ker E_1$).} 
\end{rem}
Now we use the  algorithm \red{presented} in the Appendix to implement the geometric reduction method \cyan{being a practical application of Proposition \ref{Pro:invariant manifold}} and to show that any  DAE $\Xi$  has isomorphic solutions with an ``internal'' DAE $\Xi^*$ defined on its locally maximal invariant submanifold $M^*$. In the statement of Proposition \ref{Pro:ismsol}, we refer to the submanifold $M^*=M^*_{k+1}$, the neighborhood $U^*=U^*_{k+1}$, \cyan{the coordinates $(z^*,\bar z_1,\ldots,\bar z_{k^*})$ on $U^*$,} and the   DAE $\Xi^*_{r^*, n^*}=(E^*,F^*)$ defined on $M^*$ by the algorithm of the Appendix, where  $E^*= E_{k^*+1}:M^*\to \mathbb R^{r^*\times n^*}$, $F^*= F_{k^*+1}:M^*\to \mathbb R^{r^*}$, $n^*=n_{k^*}=n_{k^*+1}$, $r^*=r_{k^*+1}$ come from Step $k^*+1$  of the algorithm. 
\begin{pro}[isomorphic solutions] \label{Pro:ismsol}
	Consider a DAE $\Xi_{l,n}=(E,F)$, fix a point $x_p\in X$.   Suppose that     {\rm  \textbf{Assumptions 1} and \textbf{2}}  of the algorithm in Appendix are satisfied. 	 
	Then  $M^c_k$, for $k=0,\ldots,k^*+1$, given by (\ref{Eq:Mseq}) of the geometric reduction method are smooth connected embedded submanifolds and  $\dim E(x)T_xM^*=const.$ for all $x\in M^*\cap U^*$. Thus by Proposition \ref{Pro:invariant manifold},   $x_p\in M^*$ is a consistent point and $M^*$ is a locally maximal   invariant submanifold  around $x_p$, given by
	$M^*= \left\lbrace x \,|\, \bar z_1(x)=0,\ldots,\bar z_{k^*}(x)=0 \right\rbrace$.
Then  for the DAE $\Xi^*_{r^*,n^*}=(E^*,F^*)$, \red{defined by the algorithm}, given \red{on $M^*$ by}
\begin{align}\label{Eq:Xi*}
\Xi^*:E^*(z^*)\dot z^*=F^*(z^*),
\end{align} 
 where $z^*=z_{k^*+1}=z_{k^*}$ are local coordinates on $M^*$, we have  ${\rm rank\,}E^*(z^*)=r^*$, $\forall z^*\in M^*$, i.e., $E^*(z^*)$ is of full row rank. 
 
 Moreover, the DAE $\Xi^*$  has isomorphic solutions with $\Xi_{l,n}$, i.e.,  \red{there exists a local diffeomorphism  $\Psi:U^*\to \Psi(U^*)$, $\Psi(x)=\hat z=(z^*,\bar z)=(z^*,\bar z_1,\ldots,\bar z_{k^*})$,} transforming the set of all solutions of $\Xi_{l,n}$ on $U^*$ into that of  $\hat \Xi_{\hat l,\hat n}=(\hat E,\hat F)$ on $\Psi(U^*)$, where $\hat l=r^*+(n-n^*)$, $\hat n=n$, given by 
	\begin{align}\label{Eq:DAEsep}
\hat\Xi:\left\lbrace  \begin{aligned}
E^*(z^*)\dot z^*=F^*(z^*),\\\bar z_1=0,\ldots, \bar z_{k^*}=0.
\end{aligned}\right. 
	\end{align}
\end{pro}
The proof is given in Section \ref{sec:prf_Thm1}. The analysis of Proposition \ref{Pro:ismsol} shows clearly the reason behind Remark \ref{Obs:1}: if we assume two DAEs $\Xi$ and $\tilde \Xi$ {to} have corresponding solutions, this assumption only gives the information that the two internal DAE $\Xi^*$ and $\tilde \Xi^*$, which \red{have} isomorphic solutions with $\Xi$ and $\tilde \Xi$, respectively, are ex-equivalent \cyan{when restricted to $M^*$ and $\tilde M^*$}, respectively, i.e.,   via  a diffeomorphism  between the  submanifolds $M^*$ and $\tilde M^*$ and  an invertible map $Q$ defined on the invariant submanifold $M^*$. We do not know, however, whether the diffeomorphism and the map $Q$ can be extended outside the submanifold $M^*$.  In fact, outside the manifolds, the two DAEs may have completely different behaviors or even different size of system matrices.  This analysis gives a motivation to introduce the concept of \emph{internal equivalence} of two DAEs (see the formal Definition~\ref{Def:in-equivalence}), which is defined by the ex-equivalence of two internal  DAEs. In Proposition \ref{Pro:ismsol}, the   internal  DAE $\Xi^*$ is defined with the help of the geometric reduction algorithm. Now we introduce two notions: local restriction and   full row rank reduction,  which can be used to define the  internal  DAE $\Xi^*$ of a DAE $\Xi$  (which we call the reduction of local $M^*$-restriction of $\Xi$, see Proposition \ref{Pro:restriction})  without going through the  algorithm when  the invariant  submanifold $M^*$   is  a priori given. The local restriction of a DAE to a   submanifold $N$ (invariant or not) is defined as follows.
\begin{defn}[local restriction]\label{Def:restriction}
	Consider a DAE $\Xi_{l,n}=(E,F)$ and   a smooth connected embedded submanifold  $N\subseteq X$ containing a point $x_p$.  Let $\psi(x)=z=(z_1,z_2)$ be local coordinates on a neighborhood $U$ of $x_p$ such that $N\cap U=\left\lbrace z_2=0 \right\rbrace $ and $z_1$ are thus coordinates on $N\cap U$.  The restriction of $\Xi$ to $N\cap U$, called local $N$-restriction of $\Xi$ and denoted  $\Xi|_{N}$, is 
		\begin{align}\label{Eq:restric}
	\Xi|_{N}:	\tilde E(z_1,0)\left[ {\begin{matrix}
			\dot z_1\\
			0
			\end{matrix}} \right]=\tilde F(z_1,0),
		\end{align} 
		where $\tilde E\circ\psi=E \left(  \frac{\partial \psi}{\partial x}\right)^{-1}  $, $\tilde F\circ\psi=F$.
\end{defn}
For any DAE $\Xi_{l,n}=(E,F)$, there may exist some redundant equations (in particular, some trivial algebraic equations $0=0$ and some dependent equations). In the linear case, we have defined the full rank reduction of a linear DAE (see Definition 6.4 of  \cite{chen2019a}). We now generalize this notion of reduction to nonlinear DAEs  to get rid of their redundant equations.
	\begin{defn}[reduction]\label{Def:rednonDAE}
		For a DAE $\Xi_{l,n}=(E,F)$, assume $\rk E(x)=const.=q$. Then there exists $Q :X\rightarrow GL(l,\mathbb R)$ such that $E_1$ of $QE=\left[ \begin{smallmatrix}
		E_1\\ 
		0
		\end{smallmatrix}\right]$  is of full row rank $q$, denote $QF=\left[ \begin{smallmatrix}
	F_1\\ 
	F_2
		\end{smallmatrix}\right]$. Assume that $\rk {\rm D} \red{F_2(x)}=const.=\hat l-q\le l-q$. Then   the full row rank reduction, shortly reduction, of $\Xi$, denoted by $\Xi^{red}$, is the DAE
		\begin{align*}
			\Xi^{red}: \left[ \begin{matrix}
			E_1(x)\\ 
			0
			\end{matrix}\right]\dot x=\left[ \begin{matrix}
			F_1(x)\\ 
			\hat F_2(x)
			\end{matrix}\right],
		\end{align*}
where $\hat F_2:X\to \mathbb R^{\hat l-q}$ with  ${\rm D} \hat F_2$ being all independent rows of ${\rm D} F_2$.
\end{defn} 
\begin{rem}
	 Clearly, since the choice of $Q(x)$ is not unique, the reduction of $\Xi$ is not unique \red{either}. Nevertheless, since $Q(x)$ preserves the solutions, each reduction $\Xi^{red}$ has the same solutions as the original DAE $\Xi$.
\end{rem} 

For a locally invariant submanifold $ M$, we consider the  local $M$-restriction $\Xi|_{M}$ of $\Xi$, and then  we construct a reduction of $\Xi|_{M}$ and denote it by $\Xi|^{red}_{M}$. Notice that the order matters: to construct $\Xi|^{red}_{M}$, we first restrict and then reduce while  reducing first and then restricting will  not give $\Xi|^{red}_{ M}$ but another DAE $\Xi^{red}|_{ M}$, \cyan{which may have redundant equations as seen from the following example}. 
\begin{exa}
Consider the following nonlinear DAE  $\Xi:\left[ \begin{smallmatrix}
1&1\\
x&0\\
0&0\\
e^{y}&e^{y}
\end{smallmatrix}\right]\left[ \begin{smallmatrix}
\dot x\\
\dot y 
\end{smallmatrix}\right]= \left[ \begin{smallmatrix}
x^2\\
x^3\\
xy\\
e^yx^2
\end{smallmatrix}\right]$ defined on $X=\mathbb R^2$. Fix a point $(x_p,y_p)=(1,0)$, then it is clear that $M^*=\left\lbrace (x,y)\in \mathbb R^2:x>0,y=0 \right\rbrace $ is a locally maximal invariant submanifold around $x_p$. \red{Set} $\psi(x,y)=(z_1,z_2)=(x,y)$ as coordinates on $X$. Then the $M^*$-restriction of $\Xi$, by Definition \ref{Def:restriction},  is  $\Xi|_{M^*}:\left[ \begin{smallmatrix}
1\\
z_1\\
0\\
1
\end{smallmatrix}\right]\dot z_1= \left[ \begin{smallmatrix}
z_1^2\\
z_1^3\\
0\\
z_1^2
\end{smallmatrix}\right]$ and the reduction of $\Xi|_{M^*}$ is $\Xi|^{red}_{M^*}:q(z_1)\dot z_1=q(z_1)z^2_1$, where $q(z_1)$ can be any non-zero function (illustrating that the reduction is not unique). \red{On the other hand,} $\Xi^{red}|_{M^*}$  is $\left[ \begin{smallmatrix}
1\\
z_1\\
0
\end{smallmatrix}\right]\dot z_1= \left[ \begin{smallmatrix}
z_1^2\\
z_1^3\\
0
\end{smallmatrix}\right]$, and clearly, has redundant equations.
\end{exa}
\begin{pro}\label{Pro:restriction}
	{Consider a DAE $\Xi_{l,n}=(E,F)$ and fix a  point $x_p$. Let $M$ be an $\bar n$-dimensional locally invariant submanifold of $\Xi$ around $x_p$. Assume that $\dim \,E(x)T_xM=const.=\bar r$ for all $x\in M$ around $x_p$. Then \cyan{any} reduction $\Xi|^{red}_M$  of the local $M$-restriction of $\Xi$  is a DAE of the form (\ref{Eq:DAE0}) and the dimensions related to $\Xi|^{red}_M$ are $\bar r$ and $\bar n$, i.e.,  $\Xi|^{red}_M=\bar \Xi_{\bar r,\bar n}$. Moreover, the matrix $\bar E$ of   $\bar \Xi_{\bar r,\bar n}=(\bar E,\bar F)$ is of full row rank $\bar r$.}
\end{pro}
\begin{proof}
	We skip the proof since we have already constructed $\Xi|^{red}_{M}$ for $M$ being an invariant submanifold, see (\ref{Eq:restrictedDAE}) in the proof of Proposition \ref{Pro:invariant submanifold}; it is clear that  $\bar E=[\bar E^1_1 \ \bar E^2_1]$, $\bar F=\bar F_1$ and $\rk \bar E=\bar r$.

\end{proof}
The definition of the internal equivalence of two DAEs is given as follows.
\begin{defn}\label{Def:in-equivalence}(internal equivalence)
	Consider two DAEs  $\Xi=(E,F)$ and $\tilde \Xi=(\tilde E,\tilde F)$, and fix two points $x_p\in X$ and $\tilde x_p\in \tilde X$. Let $M^*$ and $\tilde M^*$ be two locally maximal invariant submanifolds of $\Xi$ and $\tilde \Xi$, around $x_p$ and $\tilde x_p$, respectively.  Assume that  $\dim E(x)T_xM^*=const.$ for $x\in M^*$ around $x_p$ and  $\dim \tilde E(\tilde x)T_{\tilde x}\tilde M^*=const.$ for $\tilde x\in \tilde M^*$ around $\tilde x_p$. 
	Then, $ \Xi$ and $\tilde \Xi$ are called {locally} internally equivalent, shortly in-equivalent,  if $ \Xi|^{red}_{M^*}$ and $\tilde \Xi|^{red}_{\tilde M^*}$  are ex-equivalent, {locally around $x_p$ and $\tilde x_p$, respectively}. Denote the in-equivalence of two DAEs by $ \Xi\mathop  \sim \limits^{in} \tilde \Xi$.
\end{defn}
\begin{rem}   
	\cyan{Under the assumption that}   $\dim E(x)T_xM^*$  and  $\dim \tilde E(\tilde x)T_{\tilde x}\tilde M^*$ are constant, by Proposition~\ref{Pro:restriction} \red{applied to $M^*$}, we have $\Xi|^{red}_{M^*}=\Xi^*_{r^*,n^*}$ and  $\tilde \Xi|^{red}_{M^*}=\tilde \Xi^*_{\tilde r^*,\tilde n^*}$, where $r^*=\dim\, E(x)T_xM^*$, $n^*=\dim\, M^*$ and $\tilde r^*=\dim\, \tilde E(x)T_{\tilde x}\tilde M^*$, $\tilde n^*=\dim\, \tilde M^*$. The dimensions $l$ and $n$, related to $\Xi$, and $\tilde l$ and $\tilde n$ related to $\tilde \Xi$ are not required to be the same. However, if $\Xi$ and $\tilde \Xi$ are in-equivalent, then by definition,  $ \Xi|^{red}_{M^*}=\Xi^*_{r^*,n^*}$ and $\tilde \Xi|^{red}_{\tilde M^*}=\tilde \Xi^*_{\tilde r^*,\tilde n^*}$ are locally ex-equivalent and thus the dimensions related to them have to be the same, i.e., $r^*=\tilde r^*$ and $n^*=\tilde n^*$ \cyan{(and $l^*=r^*=\tilde r^*=\tilde l^*$ since all reductions of $\Xi$ and $\tilde \Xi$ are of full row rank).}
\end{rem}
Now {we will} study the uniqueness of solutions of DAEs with the help of the {notion of} internal equivalence (\blue{some other results of uniqueness of DAE solutions can be consulted in e.g., \cite{reich1991existence,rabier1994geometric}}).  We will say that a solution $x:I\to M^*$ \red{of a DAE $\Xi$} satisfying  $x(t_0)=x_0$, where $t_0\in I$ and $x_0\in M^*$, is maximal if for any solution $\tilde x:\tilde I\to M^*$ such that $t_0\in \tilde I$, $\tilde x(t_0)=x_0$ and $x(t)=\tilde x(t)$, $\forall t\in I\cap \tilde I$, we have $\tilde I \subseteq I$.
\begin{defn}\label{Def:in-regular}(internal  regularity)
	Consider a DAE $\Xi_{l,n}=(E,F)$ and
	let $M^*$ be a locally maximal invariant submanifold around a point $x_p\in M^*$.  Then $\Xi$ is called locally \emph{internally regular} (around $x_p$) if there exists a neighborhood $U\subseteq X$ of  $x_p$ such that for any point $x_0\in M^*\cap U$, there exists only one maximal solution $x:I\to M^*\cap U$ satisfying  $x(t_0)=x_0$	for a certain $t_0\in I$.
\end{defn}  
\begin{thm}\label{Thm:1}  
	Consider a DAE $\Xi_{l,n}=(E,F)$ and let $M^*$ be an $n^*$-dimensional locally maximal invariant submanifold around a point $x_p\in M^*$. Assume that   $\dim\, E(x)T_x{M^*}=const.=r^*$    for all $x\in M^*$ around $x_p$.  Then the following conditions are   equivalent:
	\begin{enumerate}[(i)]
		\item $\Xi$ is internally regular around $x_p$;
		\item  $\dim\, M^*=\dim\, E(x)T_xM^*$, i.e., $n^*=r^*$, for all $x\in M^*$ around $x_p$;
		\item   $\Xi$ is locally internally equivalent to 
		\begin{align}\label{Eq:ODE}
		{{\dot z^*}} = f^*\left( {{z^*}} \right),	
		\end{align}	
	\red{for $z^*\in M^*\cap U$, where $U$ is a neighborhood of $x_p$ and $f^*$ is a smooth vector field on $M^*\cap U$.}
	\end{enumerate}
\end{thm}
The proof is given in Section \ref{sec:prf_Thm1}.
\begin{rem}\label{Rem:in-eq}
	
	
 Theorem \ref{Thm:1} is a nonlinear generalization of the results on the internal  regularity of  linear  DAEs in \cite{chen2019a} (see also \cite{berger2015}, where the internal regularity is called autonomy). As stated in Theorem 6.11 of \cite{chen2019a}, a linear DAE $\Delta=(E,H)$, given by (\ref{Eq:DAE2}), is internally regular if and only if the maximal invariant subspace $\mathscr M^*$ of $\Delta$ (i.e., the largest subspace such that $H\mathscr M^*\subseteq E\mathscr M^*$) satisfies $\dim\, \mathscr M^*=\dim\, E\mathscr M^*$. {A nonlinear counterpart of the last condition is (ii) of Theorem \ref{Thm:1} and thus} $M^*$ is a {natural} nonlinear generalization of $\mathscr M^*$. \red{Observe} that $M^*$ is the limit of $M_k$ as $\mathscr V^*$ is the limit of $\mathscr V_k$, \red{defined in Remark \ref{Rem:wong}}. Moreover, we have shown in \cite{chen2019a} that the maximal invariant subspace $\mathscr M^*=\mathscr V^*$, where $\mathscr V^*$ coincides with the limit of the Wong sequence $\mathscr V_k$ \cyan{defined in Remark~\ref{Rem:wong}}. 
\end{rem}
\begin{exa}\label{Ex:ir} 
Consider a  DAE $\Xi_{6,6}=(E,F)$  with the generalized state $x=(x_1,x_2,x_3,x_4,x_5,x_6)\in X$, \red{where $X=\left\lbrace x\in \mathbb R^6:x_1\ne x_6,x_6>0 \right\rbrace $,}
\begin{align}\label{Eq:DAEex}
\left[ \begin{smallmatrix}
-\ln x_6 &x_6(x_3+x_5) &\frac{x_1x_5\ln x_6}{x_1-x_6}&0&0&0\\
0&0&0&1&0&0\\
0&0&1&0&1-\frac{x_1}{x_6}&0\\
0&0&0&0&x_5&-1\\
0&0&0&0&0&0\\
0&0&0&0&0&0
\end{smallmatrix}\right] \left[ \begin{smallmatrix}
\dot x_1\\
\dot x_2\\
\dot x_3\\
\dot x_4\\
\dot x_5\\
\dot x_6
\end{smallmatrix}\right]=\left[ \begin{smallmatrix}
(x_1-x_6)(x_3+x_5)-(x_2x_6-x_6^2-x_1)\ln x_6\\
x_5-x_2+x_6\\
(1-\frac{x_1}{x_6})(x_6^2-x_6x_2+x_4)\\
x_6+x_5(x_6^2-x_6x_2+x_4)\\
\frac{x_1}{x_6}\\
x_3+x_5
\end{smallmatrix}\right].
\end{align}
We consider $\Xi$ around a point $x_p=(0,1,0,0,0,1)$ and apply \red{to $\Xi$} the algorithm of the Appendix.  

 Step 1: We have $\rk E(x)=r_1=4$ on $U_1=X $. Since $E$ is already in the desired form, set $Q_1=I_6$ to get
$$
M_1=\left\lbrace x\in X: Q_1F(x)\in \im Q_1E(x)\right\rbrace=\left\lbrace x\in X: \frac{x_1}{x_6}=0, \  x_3+x_5=0 \right\rbrace. 
$$
It is clear that $x_p\in M_1$ and $M^c_1=M_1\cap U_1=M_1$ is a locally smooth connected embedded submanifold and $n_1=\dim M^c_1=4$. Then choose   new   coordinates $\bar z_1=(\bar x_1,\bar x_3)=(\frac{x_1}{x_6},x_3+x_5)$ and keep the \red{remaining} coordinates $ z_1=(x_2,x_4,x_5,x_6)$ unchanged. The system in new coordinates, denoted $\hat \Xi_1$, take the form
$$
\hat \Xi_1:\left[ \begin{smallmatrix}
	x_5\bar  x_3& 0&\frac{-\bar x_1x_5\ln x_6}{\bar x_1-1}&\frac{-\bar x_1\ln x_6}{x_6}&-x_6\ln x_6&\frac{\bar x_1x_5\ln x_6}{\bar x_1-1}\\
	0&1&0&0&0&0\\
	0&0&-\bar x_1&0&0&1\\
	0&0&x_5&-1&0&0\\
	0&0&0&0&0&0\\
	0&0&0&0&0&0
\end{smallmatrix}\right] \left[ \begin{smallmatrix}
	\dot x_2\\
	\dot x_4\\
	\dot x_5\\
	\dot x_6\\
	\dot {\bar x}_1\\
	\dot {\bar x}_3
\end{smallmatrix}\right]=\left[ \begin{smallmatrix}
	\bar x_3x_6(\bar x_1-1) -x_6\ln x_6(x_2-x_6-\bar x_1)\\
	x_5-x_2+x_6\\
	(1-\bar x_1)(x_6^2-x_6x_2+x_4)\\
	x_6+x_5(x_6^2-x_6x_2+x_4)\\
	\bar x_1\\
	\bar x_3
\end{smallmatrix}\right]. 
$$
By setting  $\bar z_1=(\bar x_1,\bar x_3)=0$,  we get the reduction of $M^c_1$-restriction of $\hat \Xi_1$ (see Definition \ref{Def:restriction} and \ref{Def:rednonDAE}) as
$$
\Xi_1=\hat \Xi_1|^{red}_{M^c_1}:\left[ \begin{smallmatrix}
0& 0&0&0\\
0&1&0&0\\
0&0&0&0&\\
0&0&x_5&-1
\end{smallmatrix}\right] \left[ \begin{smallmatrix}
\dot x_2\\
\dot x_4\\
\dot x_5\\
\dot x_6
\end{smallmatrix}\right]=\left[ \begin{smallmatrix}
-x_6\ln x_6(x_2-x_6)\\
x_5-x_2+x_6\\
(x_6^2-x_6x_2+x_4)\\
x_6+x_5(x_6^2-x_6x_2+x_4)
\end{smallmatrix}\right]. 
$$
Step 2: Consider the DAE $\Xi_1=(E_1,F_1)$. We have $\dim E(x)T_xM^c_1=\rk E_1(z_1)=r_2=2$ around $x_p$ (on $W_2= M^c_1\cap U_2=M^c_1$, where $U_2=U_1=X$). Set $Q_1=\left[ \begin{smallmatrix}
0& 1&0&0\\
0&0&0&1\\
1&0&0&0&\\
0&0&1&0
\end{smallmatrix}\right]$ and define $M_2$ by
$$
M_2=\left\lbrace   z_1: Q_1F_1(\bar z_1)\in \im Q_1 E_1(  z_1) \right\rbrace =\left\lbrace   z_1: x_2-x_6=0, \  x_6^2-x_6x_2+x_4=0 \right\rbrace .
$$
It is clear that $x_p\in M_2$, $M^c_2=M_2\cap U_2=M_2$ and $n_2=\dim M^c_2=2$. Then choose new coordinates $\bar z_2=(\bar  x_2, \bar x_4)=(x_2-x_6, x_6^2-x_6x_2+x_4)$ and keep the remaining coordinates $z_2=(x_5,x_6)$ unchanged. For the system in new coordinates, denoted  $\hat \Xi_2$, by a similar procedure as in Step $1$, we can define the reduction of $M^c_1$-restriction of $\hat \Xi_2$  as
$$
\begin{small} 
\Xi_2=\hat \Xi_2|^{red}_{M^c_2}:\left[ \begin{matrix}
0& 0\\
x_5&-1
\end{matrix}\right] \left[ \begin{matrix}
\dot x_5\\
\dot x_6
\end{matrix}\right]=\left[ \begin{matrix}
x_5\\
x_6
\end{matrix}\right].
\end{small}
$$
Step 3: For $\Xi_2=(E_2,F_2)$,  we have $\dim E(x)T_xM^c_2=\rk E_2( z_2)=r_2=1$ in $W_3=M^c_2$. By definition,
$M^c_3=M_3=\left\lbrace  z_2:x_5=0 \right\rbrace $.  It can be observed that $\dim M^c_3=n_3=1$ and by a similar construction as \red{at} former steps, we have 
$$
\Xi_3=\bar \Xi_2|^{red}_{M^c_3}: -\dot {x}_6=x_6.
$$
Step 4: We have $M^c_4=M^c_3$ ($\dim M^c_4=n_4=n_3=1$) and $\dim E(x)T_xM^c_4=r_4=1$, thus  $k^*=3$ \red{and  the algorithm stops at Step $k^*+1=4$}.  Therefore, by Proposition \ref{Pro:ismsol},   
 $$\cyan{M^*=M^c_4=\left\lbrace x\in \mathbb R^6:x_1=x_3=x_4=x_5,x_2=x_6,x_6>0\right\rbrace= \left\lbrace x\in \mathbb R^6:\bar x_1=\cdots=\bar x_5=0,\bar x_6>0\right\rbrace}$$ 
is locally maximal invariant and $x_p\in M^*$ is a consistent point. Moreover, since \cyan{$x_6(t)=e^{-t}x_{60}$} is the unique maximal solution of $\Xi^*=\Xi_3$ passing through \red{$x_{0}\in M^*$}, we have that $x(t)=\Psi^{-1}(x_6(t), 0, 0, 0, 0,0)=(0, e^{-t}x_{60}, 0, 0, 0,e^{-t}x_{60})$     
 is the unique maximal solution of $\Xi$ passing through $x_0=\Psi^{-1}(x_{60},0, 0, 0, 0,0)\in M^*$,  where $\Psi(x)=(x_6,\frac{x_1}{x_6},x_2-x_6,x_3+x_5,x^2_6-x_2x_6+x_4,x_5)$ is a local diffeomorphism (actually, $z^*=x_6$). Hence the DAE $\Xi$ is internally regular around $x_p$ by definition, which \red{illustrates} the results of Theorem~\ref{Thm:1} since $\dim M^*=n_4=\dim E(x)T_xM^*=r_4=1$, and $\Xi$ is   in-equivalent to the ODE: $\dot x_6=-x_6$.
\end{exa}
\section{Explicitation with driving variables of nonlinear DAEs}\label{subsection:3.2}

The explicitation (with driving variables) of a DAE $\Xi$ is the following  procedure.  
\begin{itemize}
	\item For a DAE $\Xi_{l,n}=(E,F)$, assume that ${\rm rank\,}E(x)=const.=q$ in a neighborhood $U\subseteq X$ of a point $x_p\in X$. Then there exists $Q:U\rightarrow GL(l,\mathbb{R})$ such that $Q(x)E(x)=\left[ {\begin{smallmatrix}
		{{E_1}(x)}\\
		0
		\end{smallmatrix}} \right]$,  where $E_1:U\rightarrow \mathbb{R}^{q\times n}$, and ${\rm rank\,}E_1(x)=q$.  Thus $\Xi$ is, \red{locally on $U$}, ex-equivalent via $Q(x)$ to
	\begin{align}\label{Eq:DAE6}
	\left\{ \begin{array}{c@{\ }l}
	{E_1}(x)\dot x &= {F_1}(x),\\
	0 &= {F_2}(x),
	\end{array} \right.
	\end{align}
	where $Q(x)F(x)=\left[  {\begin{smallmatrix}
		F_1(x)\\
		F_2(x)
		\end{smallmatrix}} \right] $, and where $F_1:U\to \mathbb R^{q}$, $F_2:U\to \mathbb R^{l-q}$.
	\item The matrix $E_1(x)$ is of full row rank $q$, choose its right inverse $E^{\dagger}_1(x)$, i.e., $E_1E^{\dagger}_1=I_q$ and set $f(x)=E_1^{\dagger}(x)F_1(x)$. The collection of all $\dot x$ satisfying  $E_1(x)\dot x=F_1(x)$ of (\ref{Eq:DAE6}) is given by the differential inclusion: 
	\begin{align}\label{Eq:admicon}
	\dot x\in f(x)+\ker E_1(x)=f(x)+\ker E(x).
	\end{align}
	\item Since $\ker E(x)$ is a distribution of constant rank $n-q$,  choose {locally $m=n-q$} independent vector {fields} $g_1,\ldots,g_m$ on $X$ such that $\ker E(x)={\rm span}\left\lbrace g_1,\ldots,g_m \right\rbrace(x) $. Then by introducing \emph{driving variables} $v_i$, $i=1,\ldots,m$, we  parametrize the affine distribution $f(x)+\ker E_1(x)$ and {thus} all solutions of (\ref{Eq:admicon}) {are given by} all solutions (corresponding to all controls $v_i(t)\in \mathbb R$) of
	\begin{align}\label{Eq:controlsyst} 
	\dot x=f(x)+\sum\limits_{i = 1}^m {g_i(x)}v_i.
	\end{align}
	\item    Form a matrix $g(x)=[g_1(x),\ldots,g_m(x)]$. Then, we  rewrite  equation (\ref{Eq:controlsyst}) as $\dot x = f(x) +  g(x)v$, where $v=(v_1,\ldots,v_m)$, and set $h(x)=F_2(x)$.  We claim, see Proposition \ref{Pro:exlsol} below, that all solutions of DAE (\ref{Eq:DAE6}) (and thus of the original DAE $\Xi$)  are in one-to-one correspondence with all solutions (corresponding to all $\mathcal C^0$-controls $v(t)$) of
	\begin{align}\label{Eq:proExpl}
	\left\lbrace  \begin{array}{c@{\ }l}
	\dot x &= f(x) +  g(x)v,\\
	0 &= h(x).
	\end{array} \right.
	\end{align}
 
	\item {To (\ref{Eq:proExpl}), we attach the control system $\Sigma=\Sigma_{n,m,p}=(f,g,h)$, given by} 
	\begin{align}\label{Eq:noncontrolsys}
	\Sigma:\left\lbrace \begin{array}{c@{\ } l}
	\dot x &= f(x) +  g(x)v,\\
	y &= h(x),
	\end{array} \right.
	\end{align} 
	where $n=\dim\,x$, $m=\dim\,v$, $p=\dim\,y$. Clearly, $m=n-q$ and $p=l-q$ (we will use these dimensional relations in the following discussion). In the above way, we attach a control system $\Sigma$  to a DAE $\Xi$ \red{(actually, a class of control systems, see Proposition \ref{Pro:NonDAEexpl} below)}.
\end{itemize}
\begin{defn}(explicitation with driving variables)\label{Def:expl}
	Given a DAE $\Xi_{l,n}=(E,F)$, fix a point $x_p\in X$ and assume that   $\rk E(x)=const.$  locally around $x_p$. Then, by {a} $(Q,v)$-explicitation  we will call \red{any} control system $\Sigma=\Sigma_{n,m,p}=(f,g,h)$ given by (\ref{Eq:noncontrolsys}) 
with $$f(x)=E_1^{\dagger }F_1(x), \ \ \ {\rm Im\,} g(x)=\ker {E}(x), \ \ \ h(x)=F_2(x),$$ where 
		$
		QE(x)=\left[ {\begin{smallmatrix}
			{{E_1}(x)}\\
			0
			\end{smallmatrix}} \right]$, $QF(x)=\left[ {\begin{smallmatrix}
			F_1(x)\\
			F_2(x)
			\end{smallmatrix}} \right].
		$
		The class of all $(Q,v)$-explicitations will be called shortly the explicitation class. If a particular control system $\Sigma$ {belongs} to the explicitation class of $\Xi$, we will write  $\Sigma\in\mathbf{Expl}(\Xi)$.
\end{defn}
Notice that a given $\Xi$ has many $(Q,v)$-explicitations since the construction of $\Sigma\in \mathbf{Expl}(\Xi)$ is not unique:  there is a freedom in choosing $Q(x)$, $E_1^{\dagger}(x)$, and $g(x)$. As a consequence of this non-uniqueness of construction, the explicitation  $\Sigma$ of $\Xi$ is a system defined up to a \emph{feedback transformation}, an \emph{output multiplication} and a \emph{generalized output injection} (or, {equivalently,} a class of systems).
\begin{pro}\label{Pro:NonDAEexpl}
	Assume that a control system $\Sigma_{n,m,p}=(f,g,h)$ is a $(Q,v)$-explicitation of a DAE $\Xi_{l,n}=(E,F)$ corresponding to {a} choice of invertible matrix $Q(x)$, right inverse $E_1^{\dagger}(x)$, and matrix $g(x)$. Then a control system  $\tilde \Sigma_{n,m,p}=(\tilde f,\tilde g,\tilde h)$ is a $(\tilde Q,\tilde v)$-explicitation of $\Xi_{l,n}$ corresponding to {a} choice of invertible matrix $\tilde Q(x)$, right inverse $\tilde E_1^{\dagger}(x)$, and matrix $\tilde g(x)$  if and only if $\Sigma$ and $\tilde \Sigma$ are equivalent via a $v$-feedback transformation of the form $v=\alpha(x)+\beta(x)\tilde v$, a generalized output injection $\gamma(x)y=\gamma(x)h(x)$ and an output multiplication $\tilde y=\eta(x)y$, 
	which map 
	\begin{align}\label{Eq:map}
	f\mapsto \tilde f=f+\gamma h+g\alpha, \ \ \ g\mapsto  \tilde g=g\beta, \ \ \ h\mapsto  \tilde h=\eta h,
	\end{align}
	where $\alpha$, $\beta$ and $\eta$ are smooth matrix-valued functions \red{of appropriate sizes}, \red{$\gamma=(\gamma_1,\ldots,\gamma_p)$ is a $p$-tuple of smooth vector fields on $X$,} and $\beta$ and $\eta$ are invertible.
\end{pro}
The proof is given in Section \ref{sec:prf_Thm2}. Since {the} explicitation of a DAE is a class of control systems, we {will propose now an} equivalence relation for  control systems. {An} equivalence of two nonlinear control systems is usually defined by state coordinates transformations and feedback transformations (e.g. see \cite{Isid95,nijmeijer1990nonlinear}), and sometimes output coordinates transformations \cite{marino1994equivalence}.  In the present paper, we define {a more general} system equivalence of two control systems as follows.
\begin{defn}\label{Def:sys-equivalence} (system equivalence)
	Consider two control systems $\Sigma_{n,m,p}=(f,g,h)$ and $\tilde \Sigma_{n,m,p}=(\tilde f,\tilde g,\tilde h)$ defined on $X$ and $\tilde X$, respectively. {The systems} $\Sigma$ and $\tilde \Sigma$ are called system equivalent, or shortly sys-equivalent, denoted by $\Sigma\mathop  \sim \limits^{sys} \tilde \Sigma$, if there exist a diffeomorphism $\psi:X\rightarrow \tilde X$, matrix-valued functions $\alpha:X\rightarrow\mathbb R^{m}$, $\gamma:X\rightarrow\mathbb R^{n\times p}$ and  $\beta:X\rightarrow GL(m,\mathbb R)$, {and $\eta:X\rightarrow GL(p,\mathbb R)$ such that}
	$$
	\tilde f\circ\psi =  \frac{\partial \psi}{\partial x} \left( f+\gamma h+g \alpha\right) , \ \
	\tilde g\circ\psi= \frac{\partial \psi}{\partial x}g\beta,\ \
	\tilde h\circ\psi=\eta h.
	$$
	If $\psi:U\rightarrow\tilde U$ is a local diffeomorphism between neighborhoods $U$ of $x_p$ and $\tilde U$ of $\tilde x_p$, and $\alpha$, $\beta$, $\gamma$, $\eta$ are defined locally on $U$, we will speak about local sys-equivalence.
\end{defn}
\begin{rem}\label{rem:expl}
 	The above defined sys-equivalence of two nonlinear control systems  generalizes the Morse equivalence of two linear control systems (see \cite{morse1973structural,chen2019a}). 
\end{rem}
The following proposition shows that solutions of any DAE are in \red{a} one-to-one correspondence with solutions of its $(Q,v)$-explicitation.
\begin{pro}\label{Pro:exlsol}
Consider a DAE $\Xi_{l,n}=(E,F)$ and let a control system $\Sigma_{n,m,p}=(f,g,h)$ be a $(Q,v)$-explicitation of $\Xi$, i.e., $\Sigma\in \mathbf{Expl}(\Xi)$. Then a $\mathcal C^1$-curve $x(\cdot)$ is a solution of $\Xi$ if and only if there exists $v(\cdot)\in \mathcal C^0$ such that $(x(\cdot),v(\cdot))$ is a solution of $\Sigma$ respecting the output constraints $y=0$, i.e., a solution of (\ref{Eq:proExpl}).
\end{pro}
The proof is given in Section \ref{sec:prf_Thm2}. The following theorem is a fundamental result of the present paper, which shows that sys-equivalence  for explicitation
systems (control systems) {is a true counterpart of} the ex-equivalence for DAEs.
\begin{thm}\label{Thm:ex and sys}
	Consider two DAEs $\Xi_{l,n}=(E,F)$ and $\tilde \Xi_{l,n}=(\tilde E,\tilde F)$. Assume that ${\rm rank\,} {E(x)}$ and ${\rm rank\,}\tilde E(\tilde x)$ are constant around two points $x_p$ and $\tilde x_p$, respectively. Then {for} any two control systems $\Sigma_{n,m,p}=(f,g,h)\in \mathbf{Expl}(\Xi)$ and $\tilde \Sigma_{n,m,p}=(\tilde f,\tilde g,\tilde h)\in \mathbf{Expl}(\tilde \Xi)$, {we have that locally} $\Xi\mathop  \sim \limits^{ex}\tilde \Xi$ if and only if $\Sigma\mathop  \sim \limits^{sys}\tilde \Sigma$.
\end{thm}
The proof is given in Section \ref{sec:prf_Thm2}. In order to show {how} the explicitation can be useful in the DAEs theory, we  discuss below how  the analysis of DAEs of Sections \ref{section:2} and \ref{subsection:3.1} {is} related to the {notion of} zero dynamics of nonlinear control theory. For a nonlinear control system $\Sigma_{n,m,p}=(f,g,h)$ and a nominal point $x_p$, assume $h(x_p)=0$. Recall its zero dynamics algorithm \cite{Isid95,nijmeijer1990nonlinear}.

Step 1: set $N_1=h^{-1}(0)$. Step $k$ ($k>1$): assume for some neighborhood $U_{k-1}\subseteq X$ of $x_p$, $N_{k-1}^c=N_{k-1}\cap U_{k-1}$ is a smooth   embedded and connected submanifold     such that  $x_p\in N_{k-1}^c$.  Set 
\begin{align}\label{Eq:Nseq}
{N_{k}} = \left\{ {x \in N^c_{k-1}:f(x) \in {T_x N^c_{k-1}}}+{\rm span}\{g_1(x),\ldots,g_m(x)\} \right\}.
\end{align}
  \cyan{For a control system $\Sigma=(f,g,h)$, a smooth embedded connected submanifold $N$ containing a point $x_p$ is called \emph{output zeroing} if (i) $h(x)=0$, $\forall x\in N$; (ii) $N$ is locally controlled invariant at $x_p$ (i.e., $\exists\, u: N\to \mathbb R^m$ and a neighborhood $U_p$ of $x_p$ such that $f(x)-g(x)u(x)\in T_xN, \, \forall x\in N_p\cap U_p$).  An output zeroing submanifold $N^*$ is locally maximal if for some neighborhood $U$ of $x_p$, any other output zeroing submanifold $N'$ satisfies $N'\cap U\subseteq N^*\cap U$.}

\begin{rem}\label{rem:zerodynamic}
	(i)	It is shown in \cite{Isid95} that $N_k$ is invariant under feedback transformations. Then  \cyan{consider} a control system $\tilde \Sigma=(\tilde f,\tilde g,\tilde h)$,   given by {applying} a \emph{generalized output injection} and an \emph{output multiplication} {to} $\Sigma$, i.e., $\tilde f=f+\gamma h$, $\tilde g=g$, $\tilde h=\eta h$, where $\gamma:X\rightarrow\mathbb R^{n\times p}$ and $\eta:X\rightarrow GL(p,\mathbb R)$. By $\tilde N_0=\tilde h^{-1}(0)=h^{-1}(0)$ (since $\eta(x)$ is invertible) and for
	\begin{align*}
	{\tilde N_{k}} &= \left\{ {x \in \tilde N^c_{k-1}:f(x)+\gamma h (x) \in {T_x \tilde N^c_{k-1}}}+{\rm span}\{\tilde g_1,\ldots,\tilde g_m\} (x)\right\} \\&=\left\{ {x \in \tilde N^c_{k-1}:f(x)+0 \in {T_x \tilde N^c_{k-1}}}+{\rm span}\{g_1,\ldots,g_m\}(x)\right\},
	\end{align*}
	we have $\tilde N_k=N_k$ for $k \ge 0$, which means {that} $N_k$ {of} the zero dynamics algorithm is invariant under \emph{generalized output injections} and \emph{output multiplications}. 
	
	(ii) The sequence of submanifolds $N^c_k$ of the zero dynamics algorithm is well-defined for the class $\mathbf{Expl}(\Xi)$, i.e., does not depend on the choice of $\Sigma\in \mathbf{Expl}(\Xi)$. Since by Proposition \ref{Pro:NonDAEexpl} any two systems  $\Sigma,\Sigma'\in\mathbf{Expl}(\Xi)$ are equivalent via a $v$-feedback, a generalized output injection, and an output multiplication, then by the argument in item (i) above, we have $\tilde N_k=N_k$.
	
\end{rem}
\begin{pro}\label{Pro:output zeroing vs invariant}
	Consider a DAE $\Xi_{l,n}=(E,F)$ {satisfying} ${\rm rank\,} E(x)=q=const.$ around a  point $x_p$ and  a control system $\Sigma=(f,g,h)\in \mathbf{Expl}(\Xi)$. \red{Denote $\mathcal G(x)={\rm span}\{ g_1,\ldots,g_m\}(x)$, where $g_i$, $1\le i\le m$, are the columns of $g$.} The following conditions
	\begin{itemize}
		\item [(A1)]  For $\Xi$, the submanifold $M^c_k$ of the geometric reduction method of Section \ref{section:2} is   smooth, embedded, connected  and $\dim\, E(x)T_xM^c_{k^*}=const.$  for all $x\in M^c_{k^*}$ around $x_p$,
		\item [(A2)] For $\Sigma$, the submanifold $N^c_k$ of the zero dynamics algorithm above is   smooth,   embedded, connected  and   $\dim \red{\mathcal G(x)}\cap T_{x}N^c_{k^*}=const.$ for all $x\in N^c_{k^*} $ around $x_p$ (see Proposition 6.1.1 in \cite{Isid95}), 
	\end{itemize} 
\cyan{are equivalent for each $k\ge 1$.} Assume that either (A1) or (A2) holds, then  the maximal invariant submanifold  $M^*=M^c_{k^*}$ of $\Xi$ coincides with the maximal output zeroing submanifold $N^*=N^c_{k^*}$ of  $\Sigma$. {Moreover}, $\Xi$ is internally regular (around $x_p$) if and only if $\mathcal G(x_p)\cap T_{x_p}N^*=0$ (equation (6.4) of \cite{Isid95}). 
\end{pro} 
The proof is given in Section \ref{sec:prf_Thm2}.
\begin{rem}
	By Proposition \ref{Pro:output zeroing vs invariant}, if there exists a unique $u=u(x)$ {that renders} $N^*$ output zeroing and locally maximal control invariant {for} a control system $\Sigma\in\mathbf{Expl}(\Xi) $, then the original DAE $\Xi$ is internally regular. Since the zero dynamics do  not depend on the choice of  explicitation, the internal regularity of $\Xi$  corresponds to {the fact that} the zero output {constraint} $y(t)=0$ of any control system $\Sigma\in \mathbf{Expl}(\Xi)$ can be {achieved} by a unique control $u(t)$ or, equivalently, the zero dynamics of $\Sigma$ \red{is a unique vector field on $N^*$.}
\end{rem}
The explicitation can be also used to characterize solutions of  DAEs which are not necessarily internally regular, that is, the restricted DAE $\Xi^*$, given by (\ref{Eq:Xi*}), has non-unique maximal solutions (recall that $\Xi^*$ has isomorphic solutions with the original DAE $\Xi$ by Proposition \ref{Pro:ismsol}). We now apply the explicitation method to $\Xi^*$ to have the following result.
\begin{pro}
Consider a DAE $\Xi=(E,F)$ and fix a point $x_p\in X$. Assume that the locally maximal invariant submanifold $M^*$ around $x_p$ exists and can be constructed  via the algorithm of Appendix. Then the reduction of local $M^*$-restriction of $\Xi$, denoted by $\Xi|^{red}_{M^*}$, coincides with the DAE $\Xi^*: E^*(z^*)\dot z^*=F^*(z^*)$ \cyan{of} Proposition \ref{Pro:ismsol} with $E^*(z^*)$ being of full row rank $r^*$.  We have
\begin{itemize}
	\item [(i)] A curve $z^*: I\to M^*$ is a solution of $\Xi^*$ if and only if it is an integral curve of the affine distribution $\mathcal A(z^*)=f^*(z^*)+\ker E(z^*)$, i.e., $\dot z^*(\cdot)\in \mathcal A(z^*(\cdot))$, where $f^*=(E^*)^{\dagger}F^*$.
	\item [(ii) ] $\mathcal C^1$-solutions of $\Xi^*$ are in one-to-one correspondence with those of any $(Q,v)$-explicitation $\Sigma^*\in\mathbf{Expl}(\Xi^*)$ of the form  
	$$
	\Sigma^*: z^*=f^*(z^*)+g^*(z^*)v,
	$$
\cyan{which is a control system without outputs,}	where $\im g^*=\ker E$, $g^*=(g_1^*,\ldots,g^*_{m^*})$ and $v=(v_1,\ldots,v_{m^*})$, \cyan{and $v(t)\in \mathcal C^0$}.
\item[(iii)] If $\ker E=\ker E^*$ is involutive, then $\Xi^*$ is ex-equivalent   (that is, the original DAE $\Xi$ is in-equivalent) to a semi-explicit DAE of the form
$$
\dot z^*_1=F_1(z^*_1,z^*_2),
$$
which can be seen as a control system that is not affine with respect to the control $z^*_2$.
\end{itemize}
\end{pro} 
\begin{proof}
	We omit the proof since item (i) is clear, and items (ii) and (iii) can be easily deduced by applying, respectively, the results of Proposition \ref{Pro:exlsol}   and that of Theorem \ref{Thm:EXWI} (see below) to $\Xi^*$.
\end{proof}

\section{Driving variable reducing and semi-explicit  DAEs}\label{subsection:3.3}
Now we {will} show by an example that sometimes we can reduce some of driving variables of a $(Q,v)$-explicitation.
\begin{exa}
	Consider a DAE $\Xi=(E,F)$, given by
		\begin{align*}
	\left[ {\begin{smallmatrix}
		{\sin {x_3}}&{ - \cos {x_3}}&0\\
		0&0&0
		\end{smallmatrix}} \right]\left[ {\begin{smallmatrix}
		{{{\dot x}_1}}\\
		{{{\dot x}_2}}\\
		{{{\dot x}_3}}
		\end{smallmatrix}} \right] = \left[ {\begin{smallmatrix}
		{F_1(x)}\\
		{x_1^2 + x_2^2 - 1}
		\end{smallmatrix}} \right],
	\end{align*}
where $F_1:X\rightarrow\mathbb R$ is smooth. By ${\rm rank\,} E(x)=1$, the explicitation class $\mathbf{Expl}(\Xi)$ is not empty. A control system $\Sigma\in \mathbf{Expl}(\Xi)$ is:
\begin{align*}
\red{\Sigma:}	\left\lbrace \begin{array}{c@{\ }l}
	\left[ {\begin{smallmatrix}
		{{{\dot x}_1}}\\
		{{{\dot x}_2}}\\
		{{{\dot x}_3}}
		\end{smallmatrix}} \right] &= \left[ {\begin{smallmatrix}
		{\sin {x_3}}\\
		{ - \cos {x_3}}\\
		0
		\end{smallmatrix}} \right]F_1(x) + \left[ {\begin{smallmatrix}
		0&{\cos {x_3}}\\
		0&{ - \sin {x_3}}\\
		1&0
		\end{smallmatrix}} \right]\left[ {\begin{smallmatrix}
		{{v_1}}\\
		{{v_2}}
		\end{smallmatrix}} \right],\\
	y &=  x_1^2 + x_2^2 - 1 ,
	\end{array} \right.
	\end{align*}
	where ${\left[ {\begin{smallmatrix}
			{\sin {x_3}}&{ - \cos {x_3}}&0
			\end{smallmatrix}} \right]^T}$ is  a right inverse of $E_1(x)=\left[ {\begin{smallmatrix}
		{\sin {x_3}}&{ - \cos {x_3}}&0
		\end{smallmatrix}} \right]$. 
	Now consider the last equation in the dynamics of $\Sigma$, which is $\dot x_3=v_1$. Observe that $v_1$ {acts on $\dot x_3$ only}, which implies that $v_1$ is decoupled from the other part of the dynamics. Thus, we may {get rid of} $v_1$ and regard $x_3$ as a new control. {Thus} the dynamics of $\Sigma$ become:
$$
		\left[ {\begin{smallmatrix}
			{{{\dot x}_1}}\\
			{{{\dot x}_2}}
			\end{smallmatrix}} \right] = \left[ {\begin{smallmatrix}
			{\sin {x_3}}F_1(x)\\
			{ - \cos {x_3}}F_1(x)
			\end{smallmatrix}} \right] + \left[ {\begin{smallmatrix}
			{\cos {x_3}}\\
			{ - \sin {x_3}}
			\end{smallmatrix}} \right]{v_2},
		$$
	where $x_1$ and $x_2$ are new states, $x_3$ and $v_2$ are the new control inputs.  \red{By rectifying the vector field $g_2=\cos x_3\frac{\partial }{\partial x_1}-\sin x_3\frac{\partial }{\partial x_2}$, we can reduce $v_2$ in a similar way.} We are, however, not able to reduce $v_1$ and $v_2$ simultaneously.
\end{exa}
Before giving the main result of this subsection, we formally define what we mean by ``reducing''   variables of a control system $\Sigma$.
\begin{defn}[driving variable \red{reduction}]\label{Def:reduce}
	For a control system $\Sigma_{n,m,p}=(f,g,h)$, let {$\mathcal G^{red}$} be {an} involutive sub-distribution of {constant rank $k$ of} the distribution $\mathcal G={\rm span}\left\lbrace g_1,\ldots,g_m \right\rbrace $. There exists a feedback transformation and a coordinates change such that, \red{locally}, {$\mathcal G^{red}={\rm span}\left\lbrace \frac{\partial}{\partial x^1_2},\ldots,\frac{\partial}{\partial x^k_2}\right\rbrace$ and} $\Sigma$ takes the form 
	\begin{align*}
	\left\lbrace 	\begin{array}{c@{\ }l}
	{{\dot x_1}} &= {f}_1\left( {{x_1},{x_2}} \right)+\sum\limits_{i=1}^{m-k} {{g}_1^i\left( {{x_1},{x_2}} \right)}v_1^i, \\
	\dot x_2&=v_2,\\
	y &= {h}\left( {x_1},{x_2} \right),
	\end{array}\right.	 
	\end{align*}
	where $v_2=(v^1_2,\ldots,v^{k}_2)$. We will say that $\Sigma$ can be $\mathcal G^{red}$-reduced to the following control system
	$$
	\left\lbrace 	\begin{array}{c@{\ }l}
	{{\dot x_1}} &= {f}_1\left( {{x_1},{x_2}} \right)+\sum\limits_{i=1}^{m-k} {{g}_1^i\left( {{x_1},{x_2}} \right)}v_1^i, \\
	y &= {h}\left( {{x_1},{x_2}} \right),
	\end{array}\right.   	
	$$
	where $x_2$ is a new control \red{and the reduced state $x_1$ is of dimension $n-k$}. We say that $\Sigma$ can be fully reduced if $\mathcal G^{red}=\mathcal G$.
\end{defn}
Now we connect   reducing of control systems with semi-explicit DAEs.
\begin{thm}\label{Thm:EXWI}
	For a DAE $\Xi_{l,n}=(E,F)$, the following   statements are  equivalent around a point $x_p\in X$:
	\begin{itemize}
		\item[(i)] ${\rm rank\, }E(x)=const.$   and the distribution $ \ker E(x)$ is involutive.
		\item [(ii)] $\Xi$ is locally ex-equivalent to a  semi-explicit DAE $\Xi^{SE}:	\left\lbrace {\begin{array}{c@{}l}
			\dot x_1&=F_1(x_1,x_2)\\
			0&=	F_2(x_1,x_2)
			\end{array}} \right.$.
		\item [(iii)] Any control system $\Sigma=(f,g,h)\in \mathbf{Expl}(\Xi)$  can be fully reduced.
	\end{itemize}
\end{thm}
The proof is given in Section \ref{sec:prf_Thm3}.
\begin{rem}\label{rem:EXPL1}
	(i) Observe that if $\Xi$ is ex-equivalent to  $\Xi^{SE}$, then by rewriting $x_2=\red{w}$ and choosing the output $y=F_2(x_1,w)$, we get the following    control system $\Sigma^w$ with an input $w$,
		\begin{align*} 
\Sigma^{w}:	\left\{ \begin{array}{c@{\ }l}
	{{\dot x}_1} &= {F_1}\left( x_1,w \right),\\
	y &= F_2\left(  x_1,w \right).
	\end{array} \right.
	\end{align*}
The above system $\Sigma^{w}$  {has} the same number of variables as $\Xi$. Thus $\Sigma^w$ is  an \emph{explicitation without driving variables} of $\Xi$. So there are two kinds of explicitation  for nonlinear DAEs, namely, explicitation with, or without, driving variables \cyan{(the latter is possible if and only if $\ker E$ is involutive)}.

	(ii) A linear DAE $\Delta=(E,H)$, given by (\ref{Eq:DAE2}), has always  two kinds of explicitations, since the rank of $E$ is always constant and the distribution $\mathcal G=\ker E$ is {always} involutive. The relations and differences of the two explicitations for linear DAEs are discussed in \cite{chen2021from} and Chapter 3 of \cite{chen2019geometric} (note that the explicitation without driving variables for linear DAEs is called the $(Q,P)$-explicitation there).
\end{rem}
\section{Nonlinear generalizations of the Weierstrass form}\label{subsection:3.4}
In this subsection, we will use the explicitation (with driving variables) procedure to {transform an internally regular DAE $\Xi_{l,n}=(E,F)$ with $l=n$, into  normal forms under the external equivalence}. A linear regular DAE  is always ex-equivalent (via linear transformations) to the Weierstrass form \textbf{WF} \cite{Weie68}, given by
\begin{align}\label{Eq:WF}
\textbf{WF}:\left[ {\begin{matrix}
	N&0\\
	0&I
	\end{matrix}} \right]\left[ {\begin{matrix}
	\dot z\\
	\dot z^*
	\end{matrix}} \right] = \left[ {\begin{matrix}
	I&0\\
	0&A
	\end{matrix}} \right]\left[ {\begin{matrix}
	z\\
	z^*
	\end{matrix}} \right],
\end{align}
where $N={\rm diag}\left( N_{1},\ldots,N_{m}\right) $, \cyan{with} $N_i$, $i=1,\ldots,m$  \cyan{being}  nilpotent matrices of index $\rho_i$, i.e., $N_i^{j}\ne 0$ for all $j=1,\ldots,\rho_i-1$ and $N_i^{\rho_i}= 0$.  The following theorem generalizes \red{that}  result and shows that {any internally regular nonlinear DAE (under the assumption that some ranks  are constant) is always ex-equivalent to a nonlinear Weierstrass form \textbf{NWF1} (see (\ref{Eq:NF1}) below)}. \red{Note that $\bar \phi_k$ in the algorithm of  Appendix, defined on $W_k\subseteq M^c_k$, can be considered as maps on $U_0\subseteq X$ by taking $\bar \Phi_k=\bar \varphi_k\circ \varphi_{k-1}\circ \cdots\circ \varphi_1(x)$. Then for $k\ge 1$ , set $H_k=\left[ \begin{matrix}
\bar \Phi_1&\ldots& \bar \Phi_k
	\end{matrix} \right] ^T$  and $H_0$ is empty.  \textbf{Assumption 1} of the algorithm of Appendix says that $\rk \tilde F^2_k(z_{k-1})=const.$ for $z_{k-1}\in M_k\cap U_k$. In (A1) below, we replace it by a stronger rank assumption on a neighborhood $U\subseteq X$ of $x_p$.}
\begin{thm}\label{Thm:NWF}
	Consider a   DAE $\Xi_{l,n}=(E,F)$, assume that ${\rm rank\,}E(x)=const.=q$ around a  point $x_p$. Also assume in the geometric reduction algorithm of Appendix that 
	\begin{itemize}
		\item [{\rm (A1)}]    $\rk \left[\begin{smallmatrix}
		\rD H_{k-1}\\
		\rD \tilde F^2_k
		\end{smallmatrix} \right](x) =const.$   \red{for $1\le k \le k^*$ ($H_0$ is absent) and for} all $x$ around $x_p$;
		\item [{\rm (A2)}]    $\dim E(x)T_xM^c_k=const.$   for $x\in M^c_k$ around $x_p$, $1\le k\le k^*$;
		\item [{\rm (A3)}] $l=n$ and $\dim M^*=\dim E(x) T_xM^*$, i.e., \red{$r^*=n^*$,} for all $x\in M^*$ around $x_p$.  
	\end{itemize}
	Then $\Xi$ is internally regular and there exists a neighborhood $U$ of $x_p$ such that $\Xi$ is locally  on $U$  ex-equivalent to the DAE (\ref{Eq:NF1}), represented \red{by} the nonlinear Weierstrass form 
	\begin{align}\label{Eq:NF1}
	\mathbf{NWF1}:	\left[ {\begin{matrix}
		{\begin{matrix}
			{{N_{{\rho_1}}}}&0& \cdots &0\\
		0&{{N_{{\rho_2}}}}& \ddots & \vdots \\
			\vdots & \ddots & \ddots &0\\
			0& \cdots &0&{{N_{{\rho_m}}}}
			\end{matrix}}&\vline& 0\\
		\hline
		{G\left( {z ,z^*} \right)}&\vline& I
		\end{matrix}} \right]\left[ {\begin{matrix}
		{{{\dot z }_1}}\\
		{{{\dot z }_2}}\\
		\vdots \\
		{{{\dot z }_m}}\\
		{\dot z^*}
		\end{matrix}} \right] = \left[ {\begin{matrix}
		{{z _1}}\\
		{{z _2}}\\
		\vdots \\
		{{z _m}}\\
		{{f^*}\left( {z ,z^*} \right)}
		\end{matrix}} \right] + \left[ {\begin{matrix}
		a_1+b_1\dot z^{\rho}\\
		a_2+b_2\dot z^{\rho}\\
		\vdots \\
		a_m+b_m\dot z^{\rho}\\
		0
		\end{matrix}} \right],
	\end{align}
	where $z_i=(z_i^1,\ldots,z_i^{\rho_i})$ and $z^*$ are  new coordinates,  and $\dot z^{\rho}=(\dot z_{1}^{\rho_1},\dot z_2^{\rho_2},\ldots,\dot z_m^{\rho_m})$, with $m=n-q$. The indices $\rho_i$, $1\le i\le m$, satisfy $\rho_1\le \rho_2\le \ldots \le \rho_{m}$.  
	
	More specifically,  for $1\le i \le m$, the $\rho_i\times \rho_i$ nilpotent \red{matrices} $N_{\rho_i}$ and the $\rho_i$-dimensional vector-valued functions $a_i+b_i\dot z^{\rho}$ are of the following form
		\begin{align*} 
	 N_{\rho_i} = \footnotesize{\left[ {\begin{matrix}
			0&{}&{}&{}\\
			1& 0 &{}&{}\\
			{}& \ddots & \ddots &{}\\
			{}&{}&1&0
			\end{matrix}} \right],} \ \ a_i+b_i\dot z^{\rho}\! =\! \left[ {\begin{smallmatrix}
			0\\
			a_i^1 +\sum\limits_{s=1}^{m}b^1_{i,s}\dot z^{\rho_s}_s\\
			\vdots \\
			a_i^{\rho_i-1}+\sum\limits_{s=1}^{m}b^{\rho_i-1}_{i,s}\dot z^{\rho_s}_s
			\end{smallmatrix}} \right], 
		\end{align*}
	where the  functions $a^k_i,b^k_{i,s}$ \red{satisfy $a^k_i|_{M^c_k}=b^k_{i,s}|_{M^c_k}=0$, for $1\le k\le \rho_i-1$.}
\end{thm}
The proof of Theorem \ref{Thm:NWF} is given in Section \ref{sec:prf_Thm4}.  This proof is closely related to the zero dynamics algorithm for nonlinear control systems  shown in \cite{Isid95} and the construction procedure of the above normal form is not difficult but quite tedious, {so} in order to avoid reproducing the zero dynamics algorithm, we will use some results directly from \cite{Isid95}  with small modifications. 
\begin{rem}

	(i) Assumption (A2) of Theorem \ref{Thm:NWF}  is equivalent to \textbf{Assumption 1} of the geometric reduction algorithm of Appendix. By Theorem~\ref{Thm:1}, we know that  (A3) of Theorem \ref{Thm:NWF} implies that $\Xi$ is internally regular around $x_p$.
	
	(ii) A \red{component-wise} expression of the above \textbf{NWF1} is
	$$
	\mathbf{NWF1}:\left\lbrace \begin{array}{c@{\ }l}
	0&=z^1_i, \ \ 1\le i \le m,\\
	\dot z^k_i&=z^{k+1}_i+	a_i^k +\sum\limits_{s=1}^{m}b^k_{i,s}\dot z^{\rho_s}_s,  \ \ 1\le k\le \rho_i-1,\\ 
	\dot z^*&=f^*-G\dot z,
	\end{array}\right. 
	$$
	where \red{$a^k_i,b^k_{i,s}, f^*$ and $G$ depend on $(z,z^*)$.}

	(iii)	The submanifolds   $M^c_k$, $k\ge 1$,  of the algorithm  \red{are given by}
	\begin{align*}
	M^c_k\!=\!\left\lbrace (z,z^*):z^j_i=0, \  1\le i\le m,\ 1 \le j\le k\right\rbrace,
	\end{align*}  
	\cyan{and} the maximal invariant submanifold $M^*$ is given by
	$$
	M^*=\{(z,z^*):z^j_i=0, \ 1\le i\le m, \red{1\le j\le \rho_i}\}.
	$$
Therefore, an equivalent condition for $a^k_i|_{M^c_k}=b^k_{i,s}|_{M^c_k}=0$ is that $a^k_i,b^k_{i,s}\in \mathbf{I}^k$, where $\mathbf{I}^k$ is the ideal generated by $z_i^j$, $1\le i\le m$, $1\le j\le k$ in the ring of smooth functions of $z^a_b$ and \cyan{$z^*_c$}.

	(iv) We see that \red{all maximal solutions $(z(\cdot),z^*(\cdot))$ are unique and of the form $(0,z^*(\cdot))$, where}  $z^*(\cdot)$ are maximal solutions of the  ODE $\dot z^*=f^*(0,z^*)$ on $M^*$,  {which agrees with the result of Theorem} \ref{Thm:1}(iii).
\end{rem}
 
\begin{exa}  [continuation of Example \ref{Ex:ir}]
	Consider the DAE $\Xi_{6,6}=(E,F)$ of (\ref{Eq:DAEex})  around the   point \red{$x_p=(0,1,0,0,0,1)$}.
A control system $\Sigma_{6,2,2}\in \mathbf{Expl}(\Xi)$ is
\begin{align*} 
\Sigma: 
\left[ {\begin{smallmatrix}
	\dot x_1\\
	\dot x_2\\
	\dot x_3\\
	\dot x_4\\
	\dot x_5\\
	\dot x_6\\
	\end{smallmatrix}} \right] = \left[ {\begin{smallmatrix}
	x_6(x_2-x_6)-x_1\\
	\frac{x_1}{x_6}-1\\
	0\\
	x_5-x_2+x_6\\
	x_4-x_6(x_2-x_6)\\
	-x_6
	\end{smallmatrix}} \right]  + \left[ {\begin{smallmatrix}
	x_6(x_3+x_5)&\frac{x_1x_5}{x_6}\\
	\ln x_6&0\\
	0&\frac{x_1}{x_6}-1\\
	0&0\\
	0&1\\
	0&x_5
	\end{smallmatrix}} \right]\left[ {\begin{smallmatrix}
	{{v_1}}\\
	{{v_2}}
	\end{smallmatrix}} \right],  & \ \ \left[ {\begin{smallmatrix}
	y_1\\
	y_2
	\end{smallmatrix}} \right]= \left[ \begin{smallmatrix} \frac{x_1}{x_6}\\
x_3+x_5
\end{smallmatrix}\right]. 
\end{align*}
It can be observed from Example \ref{Ex:ir} that the assumptions (A1)-(A3) of Theorem \ref{Thm:NWF} are satisfied. Now via the following local changes of coordinates defined on $U=X=\left\lbrace x\in X: x_6> 0,x_1\ne x_6 \right\rbrace $:
$$
z^1_1=\frac{x_1}{x_6}, \ \ z^2_1=x_2-x_6,\ \ z^1_2=x_3+x_5,\ \ z^2_2=x_4-x_2x_6+x^2_6, \ \ z^3_2=x_5, \ \ z^*=x_6,
$$
we can bring $\Sigma$ into \cyan{the} system   $\Sigma'$ below, which is of the zero dynamics form (\ref{Eq:ZD form}) \cyan{as given by Claim~\ref{Cl:claim1}}, 
		\begin{equation*}
		\begin{footnotesize}
\Sigma':\left\lbrace \begin{array}{c@{\,}l}
y_1&= z _1^1 \\
\dot z _1^1 &= z _1^2  +  z^1_2  v_1\\ 
\dot z _1^2 &= z^1_1+\ln z^*\cdot v_1-z^3_2v_2\\ 
y_2&= z _2^1\\
\dot z _2^1 &= z _2^2 +z^1_1v_2\\ 
\dot z _2^2&= z _2^{3} + z^*(z^1_1+\ln z^*\cdot v_1-z^3_2v_2)-z^2_1z^3_2v_2\\
\dot z _2^3 &=z^2_2+ v_2\\
\dot z^*&=-z^*+z^3_2 v_2,
\end{array} \right.  \Rightarrow \Sigma'': \left\lbrace \begin{array}{c@{\,}l}
y_1&= z _1^1 \\
\dot z _1^1 &= z _1^2 -\frac{z^1_1z^1_2}{\ln z^*}+ \frac{z^1_2}{\ln z^*}\tilde v_1+\frac{z^1_2z^3_2}{\ln z^*}\tilde v_2\\ 
\dot z _1^2 &= \tilde v_1\\ 
y_2&= z _2^1\\
\dot z _2^1 &= z _2^2 -z^1_1z^2_2+z^1_1\tilde v_2\\ 
\dot z _2^2&= z _2^{3} +z^*\tilde v_1+z^2_1z^2_2z^3_2-z^2_1z^3_2\tilde v_2\\
\dot z _2^3 &=\tilde v_2\\
\dot z^*&=-z^*-z^2_2z^3_2+z^3_2 \tilde v_2,
\end{array} \right. 
		\end{footnotesize}
	\end{equation*}
\cyan{where} the feedback transformation 
$$ 
\left[ {\begin{smallmatrix}
	\tilde v_1\\
	\tilde v_2
	\end{smallmatrix}} \right]=\left[ {\begin{smallmatrix}
	\frac{x_1}{x_6}\\
	x_4-x_2x_6+x^2_6
	\end{smallmatrix}} \right]+\left[ \begin{smallmatrix}
\ln x_6&-x_5\\
0&1
\end{smallmatrix} \right] \left[ \begin{smallmatrix}
v_1\\
v_2
\end{smallmatrix}\right], $$
\noindent \cyan{brings} the system $\Sigma'$   into the system $\Sigma''$  above. 
In order to eliminate $z^* {\tilde v_1}$ in $\dot z _2^2 = z _2^{3} +z^*\tilde v_1+z^2_1z^2_2z^3_2-z^2_1z^3_2\tilde v_2$ of $ \Sigma''$, we define the change of  coordinates 
 $$
\tilde z^1_1=z^1_1,~~\tilde z^2_1=z^2_1,~~\tilde z^1_2=z^1_2-z^*z^1_1, ~~ \tilde z^2 _2=z^2_2-z^*z^2_1,~~\tilde z^3_2=z^3_2.
$$ 
and   the output multiplication  $$\left[ \begin{smallmatrix}
\tilde 	y_1\\
	\tilde y_2
\end{smallmatrix}\right]=\left[ \begin{smallmatrix}
1&0\\
z^*&1
\end{smallmatrix}\right]\left[ \begin{smallmatrix}
y_1\\
y_2
\end{smallmatrix}\right].$$ Then the system $\Sigma''$ becomes
$$
\begin{footnotesize} 
\tilde \Sigma:\left\lbrace \begin{array}{c@{\,}l}
	\tilde y_1&= \tilde z _1^1 \\
	\dot {\tilde z} _1^1 &= \tilde z _1^2 -\frac{\tilde z^1_1(\tilde z^1_2+\tilde z^1_1z^*)}{\ln z^*}+ \frac{(\tilde z^1_2+\tilde z^1_1z^*)}{\ln z^*}\tilde v_1+\frac{(\tilde z^1_2+\tilde z^1_1z^*)\tilde z^3_2}{\ln z^*}\tilde v_2\\ 
	\dot {\tilde z} _1^2 &= \tilde v_1\\ 
	\tilde y_2&= \tilde z _2^1\\
	\dot {\tilde z} _2^1 &= \tilde z _2^2 +\tilde z^1_1(\tilde z^2_2+\tilde z^2_1z^*)({\tilde z} ^3_2-1)+\frac{\tilde z^1_1{\tilde z} ^1_2z^*}{\ln z^*}-\frac{(\tilde z^1_2+\tilde z^1_1z^*)z^*}{\ln z^*}\tilde v_1-(\tilde z^1_1\tilde z^3_2+\frac{(\tilde z^1_2+\tilde z^1_1z^*)\tilde z^3_2z^*}{\ln z^*}) \tilde v_2\\ 
	\dot {\tilde z} _2^2&=\tilde  z _2^{3} + \tilde z^2_1z^*\\
	\dot {\tilde z} _2^3 &=\tilde v_2\\
	\dot z^*&=-z^*-\tilde z^3_2(\tilde z^2_2+\tilde z^2_1z^*)+\tilde z^3_2 \tilde v_2.
\end{array} \right. 
\end{footnotesize}
$$ 
Now we drop all the tildes in the system $\tilde \Sigma$ for ease of notation. By setting $y_1=y_2=0$, replacing  $v_1=\dot z^2_1$, $v_2=z^3_2$, and deleting the  equations $\dot z^2_1=v_1$, $z^3_2=v_2$, we get the following DAE $\tilde \Xi$ from $\tilde \Sigma$,
\begin{align}
\begin{footnotesize}
\tilde \Xi: \left[ \begin{array}{cc;{2pt/2pt}ccc;{2pt/2pt}c}
0&0&0&0&0&0\\
1&0&0&0&0&0\\
\hdashline
0&0&0&0&0&0\\
0&0&1&0&0&0\\
0&0&0&1&0&0\\
\hdashline
0&0&0&0&{z}^3_2&1 
\end{array}\right]\left[ {\begin{matrix}
	\dot {z}^1_1\\
	\dot {z}^2_1\\
	\dot {z}^1_2\\
	\dot {z}^2_2\\
	\dot {z}^3_2\\
	\dot  z^*\\
	\end{matrix}} \right]=\left[ {\begin{matrix}
	{z}^1_1\\
	{z}^2_1\\
	{z}^1_2\\
	{z}^2_2\\
	{z}^3_2\\
	-  z^*-  z^3_2(  z^2_2+  z^2_1z^*)\\
	\end{matrix}} \right]+\left[ {\begin{matrix}
	0\\
	a^1_1+b^1_{11}  \dot {z}^2_1+b^1_{12} \dot {z}^3_2 \\
	0\\
	a^1_2+b^1_{21}  \dot {z}^2_1+b^1_{22}\dot {z}^3_2\\
	a^2_2 \\
	0\\
	\end{matrix}} \right],
\end{footnotesize}
\end{align}
where $a^1_1=-\frac{z^1_1(z^1_2+z^1_1z^*)}{\ln z^*}$, $b^1_{11}=\frac{(z^1_2+z^1_1z^*)}{\ln z^*}$, $b^1_{12}=\frac{(z^1_2+z^1_1z^*)z^3_2}{\ln z^*}$, $a^1_2=z^1_1(z^2_2+z^2_1z^*)({z} ^3_2-1)+\frac{z^1_1{z} ^1_2z^*}{\ln z^*}$, $b^1_{21}=-\frac{(z^1_2+z^1_1z^*)z^*}{\ln z^*} $,$b^1_{22}=z^1_1z^3_2+\frac{(z^1_2+z^1_1z^*)z^3_2z^*}{\ln z^*}$, $a^2_2=z^2_1z^*$. It is clear that $\tilde \Sigma\in \mathbf{Expl}(\tilde \Xi)$, thus we have $\Xi\mathop\sim\limits^{ex}\tilde \Xi$ since $\Sigma\in \mathbf{Expl} (\Xi)$ and $\Sigma\mathop\sim\limits^{sys}\tilde \Sigma$ (see Theorem~\ref{Thm:ex and sys}). The above DAE $\tilde \Xi$ is in the \textbf{NWF1} of (\ref{Eq:NF1}) and  the sequence of submanifolds $M^c_k$ of the geometric reduction algorithm can be expressed as $M^c_1=\left\lbrace (z,z^*):z_1^1=z^1_2=0 \right\rbrace$, $M^c_2=\left\lbrace (z,z^*)\in M^c_1:z_2^1=z^2_2=0 \right\rbrace$ and
$$
 M^*=M^c_3=\left\lbrace (z,z^*)\in M^c_2: z^3_2=0\right\rbrace .
$$
The \red{functions} $a^1_1,b^1_{11},b^1_{12},a^1_2,b^1_{21},b^1_{22}\in \mathbf{I}^1$  \red{vanish  on} $M^c_1$, and the function $a^2_2\in \mathbf{I}^2$ \red{vanishes on}  $M^c_2$.
	\end{exa} 
The form \textbf{NWF1} of Theorem \ref{Thm:NWF} is related to the zero dynamics of nonlinear control systems. In the remaining part of this section, we will use the notions of \emph{(vector) relative degree} and \emph{invariant distributions} \cyan{of} nonlinear control theory to study when a DAE $\Xi$ is ex-equivalent to a simpler form
\begin{align}\label{Eq:nonlinearWF}
\textbf{NWF2}:\left[ {\begin{matrix}
	N&0\\
	0&I
	\end{matrix}} \right]\left[ {\begin{matrix}
	\dot z\\
	\dot z^*
	\end{matrix}} \right] =  \left[ {\begin{matrix}
	z\\
	f^*(z^*)
	\end{matrix}} \right],
\end{align}
where $N={\rm diag}\left( N_{1},\ldots,N_{m}\right) $, with $N_i\in \mathbb R^{\rho_i\times \rho_i}$, $i=1,\ldots,m$,  being  nilpotent matrices of index $\rho_i$.   The \textbf{NWF2} is  \cyan{a perfect nonlinear counterpart of} the linear \textbf{WF} because the nonlinear terms $G$, $a_i$ and $b_j$ of \textbf{NWF1} \cyan{are absent in \textbf{NWF1}} and $f^*$ depends  on $z^*$-variables \cyan{only}.
We now recall the definitions of (vector) relative degree and (conditional) invariant distributions for nonlinear control systems.
\begin{defn}[relative degree \cite{Isid95}]\label{Def:r.d}
	A \cyan{square} control system $\Sigma_{n,m,m}=(f,g,h)$ has a (vector)  relative degree $\rho=(\rho_1,\ldots,\rho_m)$ at a point $x_p$ if (i) $L_{g_j}L^k_fh_i(x)=0$ for all $1\le j\le m$, $k<\rho_i-1$, for all $1\le i\le m$, and for all $x$ in a neighborhood of $x_p$;  (ii) the $m\times m$ decoupling matrix $D(x)=(L_{g_j}L^{\rho_i-1}_fh_i(x))$, $1\le i, j\le m$, is invertible around $x_p$.
\end{defn}
For a nonlinear control system $\Sigma_{n,m,p}=(f,g,h)$, define a sequence  of distributions $S_i$ by
\begin{align}\label{Eq:S_idistr}
\left\{ \begin{array}{l@{\,}l}
S_1&:=\mathcal G={\rm span}\left\lbrace g_1,\ldots,g_m\right\rbrace ,\\
S_{i+1}&:=S_i+[f,S_i\cap\ker{\rm d}h]+[\mathcal G,S_i\cap\ker{\rm d}h],\\
S^*&:=\sum\limits_{i\ge1} {S_i}.
\end{array} \right. 
\end{align} 
\begin{thm}\label{Thm:NWF2}
For a nonlinear  DAE $\Xi_{n,n}=(E,F)$ (i.e., $l=n$),  assume that $\rk E(x)=const.$ around a point $x_p\in X$. Then $\Xi$ is locally ex-equivalent to the \textbf{NWF2}, given by (\ref{Eq:nonlinearWF}), around $x_p$ if and only if there exists a control system $\Sigma=\Sigma_{n,m,m}=(f,g,h)\in \mathbf{Expl}(\Xi)$ such that 
\begin{itemize}
\item[(i)] the system $\Sigma$ has a well-defined relative degree $\rho=(\rho_1,\ldots,\rho_m)$ at $x=x_p$; 
\item[(ii)] the distributions $S_i$ of $\Sigma$, defined by (\ref{Eq:S_idistr}), are involutive for all $1\le i\le n-1$.
\end{itemize} 
\end{thm}
We omit the proof the Theorem \ref{Thm:NWF2} since it is indicated by Theorem \ref{Thm:ex and sys} and some results from nonlinear control theory, see Remark \ref{Rem:NWF2}(i) below.
\begin{rem}\label{Rem:NWF2}
(i) Note that, under conditions (i) and (ii) of Theorem \ref{Thm:NWF2}, using the results in \cite{marino1994equivalence}, we can transform \cyan{the}  system $\Sigma$ into the following form (called the input-output special form in \cite{marino1994equivalence}) via suitable coordinates transformations and  feedback transformations,
$$
\left\lbrace \begin{array}{c@{\,}l}
{{\dot z^*}} &= \bar{f}^*(z^*,y),\\
\dot z^j_{i}&=z_i^{j+1}, \ 1\le j\le \rho_i-1, \ 1\le i\le m,\\\
\dot z_i^{\rho_i}&=v_i,\\
y_i&=z_{i}.
\end{array} \right.
$$
Rewrite $\bar f^*(z^*,y)=\bar f^*(z^*,0)+\gamma(z^*,y)y$ for some smooth function $\gamma$, \cyan{then} we can always get rid of the $y$-variables  in $\bar f^*(z^*,y)$ by an output injection  $\bar f^*\mapsto \bar f^*-\gamma y=f^*$, where $f^*=f^*(z^*)$. Thus the system $\Sigma$ is always sys-equivalent to the system $\tilde \Sigma$ below
$$
\Sigma\overset{sys}{\sim}\tilde \Sigma:\left\lbrace \begin{array}{c@{\,}l}
{{\dot z^*}} &= {f}^*(z^*),\\
\dot z^j_{i}&=z_i^{j+1}, \ 1\le j\le \rho_i-1,\\
\dot z_i^{\rho_i}&=v_i, \ 1\le i\le m,\\
y_i&=z_{i}.
\end{array} \right.\ \ \ \ \overset{\text{Thm.} \, \ref{Thm:ex and sys}}{\Longleftrightarrow}  \ \ \ \ \Xi \overset{ex}{\sim} \tilde \Xi:\left\lbrace \begin{array}{c@{\,}l}
{{\dot z^*}} &= {f}^*(z^*),\\
0&=z_{i}, \ 1\le i\le m,\\
\dot z^j_{i}&=z_i^{j+1}, \ 1\le j\le \rho_i-1.\\
\end{array} \right.
$$
So by Theorem \ref{Thm:ex and sys}, the DAE $\Xi$ is ex-equivalent to $\tilde \Xi$  represented in the \textbf{NWF2} since $\tilde \Sigma\in \mathbf{Expl}(\tilde\Xi)$.

(ii) The linear \cyan{counterparts} of the  distributions $S_i$, given by (\ref{Eq:S_idistr}), for   linear control systems of the form $\Lambda:\left\lbrace \begin{smallmatrix} 
\dot x&=&Ax+Bv\\
y&=&Cv 
\end{smallmatrix}\right. $  is $\mathcal W_1=\im B$, $\mathcal W_{i+1}=A(\mathcal W_i\cap \ker C)+\im B$, and are called the conditional invariant subspaces. We have shown in \cite{chen2021from} that for a linear DAE $\Delta=(E,H)$,  if a control system $\Lambda\in \mathbf{Expl}(\Delta)$, then for all $i\ge 1$, the subspaces $\mathcal W_i$ coincides with the Wong sequences $\mathscr W_i$ of $\Delta$, given by $
\mathscr W_1=\ker E$, $\mathscr W_{i+1}=E^{-1}H\mathscr W_{i}$. Therefore, the sequences of distributions $S_i$ can be seen as a nonlinear generalization of the Wong sequence $\mathscr W_i$.

(iii) Although conditions (i) and (ii) of Theorem \ref{Thm:NWF2} are necessary and sufficient for $\Xi$ being locally ex-equivalent to \textbf{NWF2}, it is, in general, not easy to check them because the relative degree and the involutivity of distributions $S_i$ are \emph{not} invariant under output multiplications and output injections (the two properties are invariant under coordinates changes and \cyan{feedback}). From Proposition~\ref{Pro:NonDAEexpl}, we know that a control system  $\Sigma\in \mathbf{Expl}(\Xi)$ is defined up to a feedback transformation, an output multiplication and a generalized output injection. So it is possible that for one system in $\mathbf{Expl}(\Xi)$,  conditions (i) and (ii) hold while for another explicitation system the two conditions \cyan{(or one of them)} are not satisfied. The problem of finding easily checkable conditions for a DAE being ex-equivalent to the \textbf{NWF2} remains open and, in view of the above analysis, is challenging.
\end{rem} 
\section{Proofs of the results}\label{sec:Proofs}
\subsection{Proofs of Proposition \ref{Pro:invariant submanifold} and Proposition \ref{Pro:invariant manifold}}\label{sec:prf_prop}
\begin{proof}[Proof of Proposition \ref{Pro:invariant submanifold}]
	Suppose that $M$ is a locally invariant submanifold around $x_p$. By Definition \ref{Def:invariant manifold}, there exists a neighborhood $U$ of $x_p$ such that for any point $x_0\in M\cap U$, there exists a solution {$x:I\rightarrow M\cap U$} satisfying $x(t_0)=x_0$ for a certain $t_0\in I$.  Then  we have ${F}(x(t))=E(x(t))\dot x(t)\in E(x(t))T_{x(t)}M$, $\forall t\in I$. It follows that ${F}(x_0)\in E(x_0)T_{x_0}M$ by taking $t=t_0$. Hence $F(x)\in E(x)T_xM$ for all $x\in M\cap U$.
	
 	Conversely, suppose  that $\dim E(x)T_xM=const.=\bar r$ and  $F(x)\in E(x)T_xM$  locally \red{for all $x\in M\cap U$}.  Notice that $M$ is a  smooth connected embedded submanifold,  thus  there exists a \cyan{smaller} neighborhood $U_1$ of $x_p$ and local coordinates $\psi(x)=z=(z_1,z_2)$ on $U_1$ such that $M\cap U_1=\left\lbrace z_2=0\right\rbrace $, where $z_1$ are any complementary   coordinates, with $\dim z_1=\bar n$, $\dim z_2=n-\bar n$ \red{and} $\bar n=\dim M$. \red{In} the local  $z$-coordinates, the DAE $\Xi$ has the following form
		$$ 
		E(x)\left( \frac{\partial \psi (x)}{\partial x}\right) ^{-1}\left( \frac{\partial \psi (x)}{\partial x}\right) \dot x=F(x) \Rightarrow\left[ \begin{matrix}
		\tilde E_1(z)&\tilde E_2(z)
		\end{matrix} \right]\left[  {\begin{matrix}
			{\dot z}_1\\
			{\dot z}_2
			\end{matrix}} \right]=\tilde F(z),
		$$
		where $\tilde E_1:U_1\rightarrow\mathbb{R}^{l\times \bar n}$, $\tilde E_2:U_1\rightarrow\mathbb{R}^{l \times(n-\bar n)} $, $\left[ \begin{matrix}
		\tilde	E_1\circ\psi &\tilde E_2\circ\psi
		\end{matrix} \right]=	E \left( \frac{\partial \psi  }{\partial x}\right) ^{-1}$ and  $\tilde F\circ\psi=F$. By setting $z_2=0$, we consider the following DAE defined locally on $M$ (denoted by $\Xi|_{M}$ \red{and} called the local $M$-restriction of $\Xi$, see Definition \ref{Def:restriction}): 
		$$
		\Xi|_{M}:	\tilde E_1(z_1,0)\dot z_1=\tilde F(z_1,0).
		$$ 
		Then by $\dim\, E(x)T_xM=const.=\bar r$   for all $x\in M$ around $x_p$, there exists a neighborhood $U_2\subseteq U_1$ of $x_p$ such that $\rk \tilde E_1(z_1,0)=\bar r$, $\forall z_1\in M\cap U_2$. So there exists $Q:M\cap U_2\rightarrow GL\left(l,\mathbb{R}\right)$ such that 	$\bar E_1(z_1)$ of $
		Q(z_1)\tilde E_1(z_1,0)= \left[ {\begin{smallmatrix}
			\bar E_1(z_1)\\
			0
			\end{smallmatrix}} \right] $ is of full row rank $\bar r$. Rewrite $\bar E_1(z_1)\dot z_1=\left[ {\begin{smallmatrix}
			\bar	E^1_1(z_1)&\bar E^2_1(z_1)
			\end{smallmatrix}} \right]\left[ {\begin{smallmatrix}
			\dot z^1_1\\\dot z^2_1
			\end{smallmatrix}} \right]$,  where   $z_1=(z^1_1,z_1^2)$, $\bar E^1_1:M\cap U_2\rightarrow \mathbb{R}^{\bar r\times \bar r}$ and $\bar E^2_1:M\cap U_2\rightarrow\mathbb{R}^{\bar r\times {(\bar n-\bar r)}}$
		and denote $Q(z_1)\tilde F(z_1,0)=\left[ {\begin{smallmatrix}
			\bar F_1(z_1)\\
			\bar F_2(z_1)
			\end{smallmatrix}} \right]$.  Without loss of generality, we  assume that   $\bar E^1_1(z_1)$ is invertible  (if not, we  permute the components of $z_1$ such that the first \red{$\bar r$} columns of $\bar E_1(z_1)$ are independent). Now by the assumption that $F(x)\in E(x)T_{x}M$  for all $x\in M$ around $x_p$, there exists a neighborhood $U_3\subseteq U_2$ such that   $\tilde F(z)\in \tilde{E}(z)T_z\psi(M)$ for all $z\in M\cap U_3$, i.e.,
		\begin{align*}
		\left[ {\begin{matrix}
			{\bar F}_1(z_1)\\
			{\bar F}_2(z_1)
			\end{matrix}} \right] \in \im \left[ {\begin{matrix}
			\bar E^1_1(z_1)&\bar E^2_1(z_1)\\
			0&0
			\end{matrix}} \right].
		\end{align*}
		It follows that  ${\bar F}_2(z_1)\equiv 0$ for all $z_1\in M\cap U_3$. Then consider the following DAE (which is actually a reduction of $\Xi|_{M}$, \red{denoted} by $\Xi^{red}_{M}$, see Definition \ref{Def:rednonDAE})
		\begin{align}\label{Eq:restrictedDAE}
		\Xi|^{red}_{M}:	\left[ {\begin{matrix}
			\bar E^1_1(z_1)&\bar E^2_1(z_1)
			\end{matrix}} \right]\left[ {\begin{matrix}
			{\dot z_1^1}\\
			{\dot z_1^2}
			\end{matrix}}\right]  = 
		\bar F_1\left( z_1\right) .
		\end{align} 
		Note that  a $\mathcal C^1$-curve $z_1:I\to M\cap U_3$ is a solution of	(\ref{Eq:restrictedDAE}) passing through $z_{10}=(z^1_{10},z^2_{10})$ if and only if $x(\cdot)=\psi^{-1} (z_1(\cdot),0)$ is a solution of $\Xi$ passing through $x_0=\psi^{-1}(z_{10},0)$. Observe that for any initial point $z_{10}\in M\cap U_3$, there always exists a solution $z_1(\cdot)$ of (\ref{Eq:restrictedDAE})  such that $z_1(t_0)=z_{10}$ for a certain $t_0\in I$  and $z_1(t)\in M\cap U_3$, $\forall t\in I$. Indeed,  rewrite   DAE (\ref{Eq:restrictedDAE}) as the following ODE (recall that $\bar E^1_1(z_1)$ is invertible): 
	\begin{align}\label{ODE0}
	\dot z_1^1= {\left(\bar E^1_1(z_1) \right)^{ - 1}}\left( \bar F_1\left( z_1 \right) - \bar E^2_1(z_1)\dot z_1^2\right).
	\end{align}
	It is always possible to  \red{parameterize}     solutions $z_1(\cdot)=(z^1_1(\cdot),z^2_1(\cdot))$ \red{of} (\ref{ODE0}) \red{as follows.} Denote $\dot z_1^2=v$,  $f(z_1) = {(E^1_1)^{ - 1}}\tilde F_1(z_1)$ and $g(z_1) =(E^1_1)^{ - 1}E^2_1(z_1)$, then   (\ref{ODE0}) can be expressed as
	\begin{align}\label{Eq:control0}
	\left\{ \begin{array}{l}
	{{\dot z}^1_1} = f(z_1) + g\left( z_1  \right)v,\\
	{{\dot z}^2_1} = v,
	\end{array} \right.
	\end{align}
\red{(called a $(Q,v)$-explicitation of (\ref{Eq:restrictedDAE}), see Definition \ref{Def:expl}), and for any solution $(z_1(\cdot),v(\cdot))$ of (\ref{Eq:control0}), with $v\in \mathcal C^0$, the curve $z_1(\cdot)$ is a $\mathcal C^1$-solution of (\ref{Eq:restrictedDAE}) satisfying $z_1(t_0)=z_{10}$ (see Proposition \ref{Pro:exlsol}). It follows that for any point $ x_0 = \psi^{-1}(z_{10}, 0)\in M \cap U_3$, there always exists a solution $x(\cdot)= \psi^{-1}(z_1(\cdot),0)$ of  $\Xi$  such that $x(t_0) = x_0$ for a certain
$t_0\in I$ and that $x(t) \in M \cap U_3$ for all $t\in I$, so} $M$ is a locally invariant submanifold of $\Xi$ around $x_p$ by definition. 
\end{proof}
\begin{proof}[Proof of Propostion \ref{Pro:invariant manifold}]
 Let $k$ be the smallest integer such that $M^c_0\supsetneq M^c_1\supsetneq \cdots \supsetneq M^c_k$, where $M^c_i$, $0\le i\le k$ are connected embedded submanifolds, and either $x_p\notin M_{k+1}$ or $x_p\in M_{k+1}$ and $M^c_{k+1}=M_{k+1}\cap U_{k+1}$ is a submanifold (by the recursive procedure assumptions) such that $\dim M^c_k=\dim M^c_{k+1}$. Then $k^*=k$ is the integer whose existence is  \cyan{asserted}.	The condition $k^*\le n$ follows from $\dim M^c_{i-1}>\dim M^c_i $, $1\le i\le k^*$.

\emph{Claim.}  If an consistent point  $x_c\in S_c\cap U_{k^*}$, then $x_c\in M_{k^*+1}$. Now we prove the  \emph{Claim} holds.
 Since $x_c$ is consistent, there exists a solution $(x(t),u(t))$, \red{defined on $I$,}  and $t_0\in I$ such that $x(t_0)=x_c$.  It follows  that for all $t\in I$, 
	\begin{align}\label{Eq:sol}
	E(x(t))\dot x(t)=F(x(t)).
	\end{align} 
	So  $F(x(t))\in {\rm Im\,} E(x(t))$, $\forall t\in I$. Thus by equation (\ref{Eq:Mseq}), we have $x(t)\in M_1$, $\forall t\in I$. Suppose that for a certain $i>1$, we have $x(t)\in M_{i-1}$, $\forall t\in I$. We then have that $\dot x(t)\in T_{x(t)}M_{i-1}$, $\forall t\in I$ (note that {when restricted to} $U_{i-1}$, the set $M_{i-1}$ is a submanifold). Thus in $U_{k^*}\subseteq U_i$,    equation (\ref{Eq:sol}) implies  $F(x(t))\in E(x(t))T_{x(t)}M^c_{i-1}$. It follows that $x(t)\in M_i\cap U_{i-1}$, for any $t\in I$, due to (\ref{Eq:Mseq}). By an induction argument, we conclude that $x(t)\in M_{k^*+1}\cap U_{k^*}$, and, in particular, we have $x_c=x(t_0)\in M_{k^*+1}\cap U_{k^*}$. 

(i) If $x_p\in M_{k^*+1}$, we have $\dim M^c_{k^*+1}=\dim M^c_{k^*}$ and since $M^c_{k^*+1}\subseteq M^c_{k^*}$, it follows that there exists an open neighborhood $U_{k^*+1}$ such that $M^c_{k^*+1}\cap U_{k^*+1}=M^c_{k^*}\cap U_{k^*+1}$. By assumption, $M^*=M^c_{k^*+1}\cap U^*$ satisfies $\dim E(x)T_xM^*=const.$ in $U^*\subseteq U_{k^*+1}$. So, using Proposition \ref{Pro:invariant submanifold}, we conclude that $M^*$ is a locally invariant submanifold on $U^*$. To prove \red{that} $M^*$ is maximal in $U^*$, let $M'$ be any invariant submanifold, then any point $x_0\in M'\cap U^*$ is consistent,   so $x_0\in S_c\cap U^*$, then by \red{the above} \emph{Claim},    $x_0\in M_{k^*+1}\cap U^*=M^*\cap U^*$ showing that $M^*$ is maximal in $U^*$.

(ii)	 We now prove that $M^*$   coincides with the consistent set $S_c$ on $U^*$.  Since $M^*\cap U^*$ is locally  invariant,  for any point $x_{0}\in M^*\cap U^*$, there exist  at least one solution $(x(\cdot),u(\cdot))$ \red{on $I$}  and $t_0\in I$ such that $x(t_0)=x_0$, which implies that  $x_0$ is consistent i.e., $x_0\in S_c$. It follows that   $M^*\cap U^*\subseteq S_c\cap U^*$. Conversely, consider any point $x_0\in S_c\cap U^*$, using again the above \emph{Claim}, we conclude that $x_0\in M_{k^*+1}\cap U^*=M^*\cap U^*$,  which implies   $S_c\cap U^*\subseteq M^*\cap U^*$.  Therefore, $M^*\cap U^*=S_c\cap U^*$.
\end{proof}
\subsection{ Proofs of  Proposition \ref{Pro:ismsol} and Theorem \ref{Thm:1}}\label{sec:prf_Thm1}

\begin{proof}[Proof of Proposition \ref{Pro:ismsol}]
	At every Step $k$ of the algorithm in Appendix, consider the DAE $\tilde \Xi_k=\Xi_{k-1}=(E_{k-1}, F_{k-1})$ and $  \hat \Xi_k=( \hat E_k, \hat F_k)$, the latter  given by (\ref{Eq:tildeXiuk}). Then we show that the following items are equivalent.
	(a). $z_{k-1}(\cdot)=\psi^{-1}_{k}(z_k(\cdot),\bar z_k(\cdot))$   is a solution of $\Xi_{k-1}$;
	(b). $(z_k(\cdot),\bar z_k(\cdot))$ is a solution of $\hat \Xi_k$;
	(c). $\bar z_k(\cdot)=0$ and $z_k(\cdot)$ is a solution of 
	$
	\Xi_{k}:E_{k}(z_k)\dot z_k=F_{k}(z_k)
	$,
	where $E_{k} (z_k)=\hat{E}^1_{k}(z_k,0)$, $F_{k} = \hat {F}^1_{k}(z_k,0)$ and where $\hat{E}^1_{k}$, $\hat {F}^1_{k}$ are defined in (\ref{Eq:tildeXiuk}).
	Since $\tilde \Xi_k=\Xi_{k-1}$	is locally ex-equivalent to $\hat \Xi_k$ via $Q_k$ and $\psi_k$,  we have that item (a) and item (b)   above are equivalent (see Remark \ref{Obs:1}). The equivalence of item (b) and item (c) follows from the fact that the solutions exists on $M^c_k$ only and should respect  the constraints $\bar z_k=0$.
	
	Then by the equivalence of (c) and (a),  we have, at the first step of the algorithm, that  $(z_1(\cdot),0)$  is a solution of   
	$E_1(z_1)\dot z_1=F_1(z_1)$, \red{together with} $\bar z_1=0,$
	if and only if $z_0(\cdot)=\psi^{-1}_{1}(z_1(\cdot),0)$  is a solution of $\Xi_0=\Xi=(E,F)$. In general, by  an induction argument, we can prove that $(z_k(\cdot),0,\ldots,0)$  is a solution of   
	$E_k(z_k)\dot z_k=F_k(z_k)$, \red{together with} $
	\bar z_1=0,\ldots, \bar z_{k}=0$, if and only if $x(\cdot)$   is a solution of   $\Xi$, where  $x(\cdot)$ is given by the following iterative formula
	\begin{align}\label{Eq:a}
	x(\cdot)=z_0(\cdot)= \psi^{-1}_{1}(z_1(\cdot),0), \ z_1(\cdot)=\psi^{-1}_{2}(z_2(\cdot),0),\ \ldots, \ z_{k-1}(\cdot)=\psi^{-1}_{k}(z_k(\cdot),0). 
	\end{align} 
	 Each diffeomorphism $\psi_k$ is defined on $W_k$, we extend \red{it} to $U_k$ by putting $\Psi_k=(\psi_k,\bar z_k,\ldots,\bar z_1)$. Now  we define the local diffeomorphism  $\Psi:=\Psi_{k^*}\circ\cdots\circ\Psi_2\circ\Psi_1:U_{k^*+1}\to \mathbb R^n$ (note that $\Psi_{k^*+1}=\Psi_{k^*}$). To show that the local diffeomorphism $\hat z=\Psi(x)$, where $\hat z=(z^*,\bar z)$, transforms solutions of $\Xi^u$ into those of $\hat\Xi^{\hat u}$, it is enough to observe that $\Psi$ satisfies (\ref{Eq:a}), for $k=k^*+1$. Now we prove that $E^*(z^*)$, for $z^*\in M^*$, is of full row rank. 	 Consider Step $k^*+1$ of the algorithm, note that the  $Q_{k^*+1}$-transformation ensures that $\tilde E^1_{k^*+1}(z_{k^*})$ is of full row rank. By   $M^c_{k^*+1}=\left\{ {z_{k^*} \in M^c_{k^*}\cap U_{k^*+1}\,|\, {\tilde F_{k^*+1}^2(z_{k^*}) = 0} } \right\}$ and   the fact that   $\dim M^c_{k^*}=n_{k^*}=n_{k^*+1}=\dim M^c_{k^*+1}$, we have    $\tilde F_{k^*+1}^2(z_{k^*})= 0$, $\forall z_{k^*}\in M^c_{k^*}\cap U_{k^*+1}$. As a consequence, the $\bar z_{k^*+1}$-coordinates  {are not present}, so there is no equation $\bar z_{k^*+1}=0$ in (\ref{Eq:DAEsep}). Moreover, we have $M^c_{k^*+1}=M^c_{k^*}$ in $U_{k^*+1}$, implying that $z_{k^*+1}=z_{k^*}$. Finally, it is seen from $E^*(z^*)= E_{k^*+1}(z_{k^*+1})=\hat E^1_{k^*+1}(z_{k^*})=\tilde  E^1_{k^*+1}(z_{k^*})$ that $E^*(z^*)$ is of full row rank for all $ z^*=z_{k^*+1}\in M^*{=M^c_{k^*+1}}$.  
\end{proof}
\begin{proof}[Proof of Theorem \ref{Thm:1}]
	Since  $M^*$ is locally invariant around $x_p$,  via a similar construction \red{to that} shown in the proof of Proposition \ref{Pro:invariant submanifold}, we can get a DAE $\Xi|^{red}_{M^*}$   of the  form (\ref{Eq:restrictedDAE}) (if the maximal invariant submanifold $M^*$ is constructed  via the algorithm in Appendix, then $\Xi|^{red}_{M^*}$ coincides with the DAE $\Xi^*$ of (\ref{Eq:Xi*}) from the results of that algorithm). Note that $\Xi|^{red}_{M^*}$ can be seen as an ODE possibly  with  free variables (see (\ref{ODE0}) and (\ref{Eq:control0})), and that $\Xi|^{red}_{M^*}$ has isomorphic solutions with $\Xi$ (see Proposition \ref{Pro:ismsol}). Thus $\Xi$ is internally regular around $x_p$, i.e., there exists  only one maximal solution passing through any $x_0\in M^*$ around $x_p$  if and only if no free variables are present in $\Xi^*=\Xi|^{red}_{M^*}$, i.e., $[\bar E^1_1,\bar E^2_1]$ of (\ref{Eq:restrictedDAE}) is invertible  or, equivalently, $n^*=\dim M^*=\dim E(x)T_xM^*=r^*$ for all $x\in M^*$ around $x_p$ (i.e., $E^*$ of (\ref{Eq:Xi*}) is invertible). Moreover, it is clear that $[\bar E^1_1,\bar E^2_1]$    is invertible if and only if $\Xi|^{red}_{M^*}$ of (\ref{Eq:restrictedDAE}) (or $\Xi^*$, given by (\ref{Eq:Xi*})) is ex-equivalent  to an ODE (\ref{Eq:ODE}) without free variables,  where $f^*=[\bar E^1_1,\bar E^2_1]^{-1}\bar F_1$ (or $f^*=(E^*)^{-1}F^*$), that is, $\Xi$ is internally equivalent to (\ref{Eq:ODE}) around $x_p$.  
\end{proof} 
\subsection{Proofs of Proposition \ref{Pro:NonDAEexpl}, Proposition \ref{Pro:exlsol}, Theorem \ref{Thm:ex and sys} and Proposition \ref{Pro:output zeroing vs invariant}}\label{sec:prf_Thm2}
\begin{proof}[Proof of Proposition \ref{Pro:NonDAEexpl}]
	\emph{If.} Throughout the proof below, we may drop the argument $x$ for the functions  $f(x),g(x),h(x),\ldots$,   for ease of notation. Suppose that $\Sigma$ and $\tilde \Sigma$ are equivalent via transformations given {by} (\ref{Eq:map}). First, ${\rm Im\,}\tilde g={\rm Im\,} g\beta =\ker E_1=\ker E$  \red{implies} that $\tilde g$ is another choice such that ${\rm Im\,} \tilde g=\ker E$. Moreover, we have 
	$$
	\tilde \Sigma:\left\{ {\begin{array}{c@{\ }l}
		\dot x&=\tilde f+\tilde g\tilde v=f+g\alpha+\gamma h+g\beta v=E^{\dagger}_1F_1+g\alpha+\gamma F_2+g\beta v,\\
		\tilde y&=\tilde h= \eta h.
		\end{array}}\right.	
	$$
	\red{Pre-multiplying} the differential part $\dot x=E^{\dagger}_1F_1+g\alpha+\gamma F_2+g\beta v$ of $\tilde \Sigma$ by $E_1$, we get (note that ${\rm Im\,}g=\ker E_1$)
	$$
	\left\{ {\begin{array}{c@{\ }l}
		E_1\dot x&=F_1+E_1\gamma F_2, \\
		\tilde y&=\eta h.
		\end{array}}\right.	
	$$
	Thus $\tilde \Sigma$ is an $(I,\tilde v)$-explicitation of the following DAE:
	$$
	\left[ {\begin{matrix}
		E_1\\
		0
		\end{matrix}} \right]\dot x=\left[ {\begin{matrix}
		F_1+E_1\gamma F_2\\
		\eta F_2
		\end{matrix}} \right].
	$$
	Since the above DAE can be {obtained} from $\Xi$ via $\tilde Q=Q'Q$, where $Q'=\left[ {\begin{smallmatrix}
		I_q&E_1\gamma\\
		0&\eta
		\end{smallmatrix}}\right]$, it proves that $\tilde \Sigma$ is a $(\tilde Q,\tilde v)$-explicitation of $\Xi$ corresponding to the choice of invertible matrix $\tilde Q=Q'Q$. Finally, by $E_1\tilde f=F_1+E_1\gamma F_2$, we get $\tilde f=\tilde E^{\dagger}_1(F_1+\gamma F_2)$ for the \cyan{above} choice of right inverse $\tilde E^{\dagger}_1$ of $E_1$.
	
	\emph{Only if.} Suppose that $\tilde \Sigma\in \mathbf{Expl}(\Xi)$ via $\tilde Q$, $\tilde E^{\dagger}_1$ and $\tilde g$. First, by ${\rm Im\,}\tilde g=\ker E={\rm Im\,}g$, there exists an invertible matrix $\beta$ such that $\tilde g=g\beta$. Moreover, since $E^{\dagger}_1$ is a right inverse of $E_1$ if and only if any solution $\dot x$ of $E_1\dot x=w$ is given by $E^{\dagger}_1w$, we have $E_1E_1^{\dagger}F_1=F_1$ and $E_1\tilde E_1^{\dagger}F_1=F_1$. It follows that $E_1(\tilde E^{\dagger}_1-E^{\dagger}_1)F_1=0$, so $(\tilde E^{\dagger}_1-E^{\dagger}_1)F_1\in \ker E_1$. Since $\ker E_1={\rm Im\,}g$, it follows that $(\tilde E^{\dagger}_1-E^{\dagger}_1)F_1=g
	\alpha$ for {a} suitable $\alpha$. Furthermore, since $Q$ is such that $E_1$ of
	$
	QE=\left[ {\begin{smallmatrix}
		E_1\\
		0
		\end{smallmatrix}} \right]$ is of full row rank, any other $\tilde Q$, such that $\tilde E_1$ of $\tilde QE=\left[ {\begin{smallmatrix}
		\tilde E_1\\
		0
		\end{smallmatrix}} \right]$ is \cyan{of} full row rank, must be of the form $\tilde Q=Q'Q$, where $Q'=\left[ {\begin{smallmatrix}
		Q_1&Q_2\\
		0&Q_4
		\end{smallmatrix}}\right]$.
	Thus via $\tilde Q$, $\Xi$ is ex-equivalent to
	\begin{align*}
	Q'\left[ {\begin{matrix}
		E_1\\
		0
		\end{matrix}} \right]\dot x=Q'\left[ {\begin{matrix}
		F_1\\
		F_2
		\end{matrix}} \right]\Rightarrow \left[ {\begin{matrix}
		Q_1E_1\\
		0
		\end{matrix}} \right]\dot x=\left[ {\begin{matrix}
		Q_1F_1+Q_2F_2\\
		Q_4F_2
		\end{matrix}} \right].
	\end{align*} 
	The equation on the right-hand side of the above can be expressed (using $\tilde E_1^{\dagger}$ and $\tilde g$) as:
	\begin{align*}
	\left\{ {\begin{array}{*{20}{l}}
		\dot x=\tilde E_1^{\dagger}F_1+\tilde E_1^{\dagger}Q^{-1}_1Q_2F_2+\tilde gv=E^{\dagger}_1F_1+g\alpha+E_1^{\dagger}Q^{-1}_1Q_2h+g\beta\tilde v,\\
		0 = Q_4F_2=Q_4h.
		\end{array}}\right.	
	\end{align*}
	Thus the explicitation of $\Xi$ via $\tilde Q$, $\tilde E^{\dagger}_1$ and $\tilde g$ is 
	$$
	\tilde \Sigma:\left\{ {\begin{array}{*{20}{l}}
		\dot x=E^{\dagger}_1F_1+g\alpha+\gamma h+g\beta\tilde v=f + \gamma h +g(\alpha +\beta\tilde v)=\tilde f+\tilde g\tilde v,\\
		\tilde y= \eta h=\tilde h.
		\end{array}}\right.	
	$$
	where $\gamma=E_1^{\dagger}Q^{-1}_1Q_2$, $\eta=Q_4$. Therefore, we can see {that} $\Sigma$ and $\tilde \Sigma$ are equivalent via the transformations  \red{of the form} (\ref{Eq:map}).
\end{proof}
\begin{proof}[Proof of Proposition \ref{Pro:exlsol}]
Consider the DAE (\ref{Eq:DAE6})  of the $(Q,v)$-explicitation procedure. Since $Q$-transformations preserve solutions of $\Xi$,   (\ref{Eq:DAE6}) resulting from a  $Q$-transformation of $\Xi$ has the same solutions as $\Xi$. Thus we need to prove that   (\ref{Eq:DAE6}) and  (\ref{Eq:proExpl}) have corresponding solutions for any choices of $E^{\dagger}_1$ and $g$. Moreover, the second equation $0=F_2(x)$ of (\ref{Eq:DAE6}) coincides  with $0=h(x)$ of (\ref{Eq:proExpl}).  So we only need to prove that  $x(t)\in \mathcal C^1$  is a solution of $E_1(x)\dot x=F_1(x)$
  if and only if there exists $v(t)\in\mathcal C^0$ such that $(x(t),v(t))$ is a solution of $\dot x=f(x)+g(x)v$  independently of the choice of $E^{\dagger}_1$, defining  $f(x)=E^{\dagger}_1(x)F_1(x)$,  and of the choice of $g$ satisfying ${\rm Im\,}g(x)=\ker E_1(x)$. 

\emph{If.} Suppose that  $(x(t),v(t))$ is a solution of $\dot x=f(x)+g(x)v$. Then we have $
\dot x(t)=f(x(t))+g(x(t))v(t)$.
Pre-multiplying the latter equation by $E_1(x(t))$, we get  that $$E_1(x(t))\dot x(t)=E_1(x(t))f(x(t))=E_1(x(t))E^{\dagger}_1(x(t))F_1(x(t))=F_1(x(t)),$$ which proves that \cyan{$x(t)$} is a solution of $E_1(x)\dot x=F_1(x)$. 

\emph{Only if.} Suppose that $x(t)$ is a solution of $E_1(x)\dot x=F_1(x)$. Rewrite $E_1(x)\dot x$ as 
$
\left[ \begin{smallmatrix}
E^1_1(x)&E^2_1(x)
\end{smallmatrix}\right] \left[ \begin{smallmatrix}
\dot x_1\\\dot x_2
\end{smallmatrix}\right] 
$, where $E^1_1: U\to \mathbb R^{q\times q}$ is smooth and $x= (x_1,x_2) $. Then, by taking a smaller neighborhood $U$, if necessary, we assume that   $E^1_1(x)$ is invertible locally around $x_p$ (if not, we permute the components of $x$ such that the first $q$ columns of $E_1(x)$ is  independent). Thus  a choice of  right inverse of $E_1$ is $E_1^{\dagger}= \left[ \begin{smallmatrix}
(E^1_1)^{-1}\\0
\end{smallmatrix}\right] $. So the maps $f$ and $g$ can be defined as
$
f:=E^{\dagger}_1 F_1=\left[ \begin{smallmatrix}
(E^1_1)^{-1}F_1\\0
\end{smallmatrix}\right]$, $
g:=\left[ \begin{smallmatrix}
-(E^1_1)^{-1}E_2\\I_{m}
\end{smallmatrix}\right]$. 
Set $v(t)=\dot x_2(t)$,  then $v\in \mathcal C^0$ and it is clear that if $x(t)=((x_1(t),x_2(t)))$ is a solution of $E_1(x)\dot x=F_1(x)$, then $(x(t), v(t))$ solves $\dot x=f(x)+g(x)v$   since
\begin{align*}
\left[ \begin{matrix}
E^1_1(x(t))&E^2_1(x(t))
\end{matrix}\right] \left[ \begin{matrix}
\dot x_1(t)\\\dot x_2(t)
\end{matrix}\right] =F_1(x(t)) \Rightarrow \dot x_1(t)= (E^1_1)^{-1}F_1(x(t)) -(E^1_1)^{-1}E^2_1(x(t))\dot x_2(t).
\end{align*}
Notice that if we choose another right inverse $\tilde E^{\dagger}_1$ of $E_1$ and another matrix $\tilde g$ such that ${\rm Im\,}\tilde g=\ker E_1$, then by Proposition \ref{Pro:NonDAEexpl},  we have
$$
\dot x=\tilde f(x)+\tilde g(x)\tilde v  \Leftrightarrow \dot x=  f(x)+   g(x)(\alpha(x)+\beta(x)v).
$$
We thus conclude that there exists $\tilde v(t)=\alpha(x(t))+\beta(x(t))v(t)=\alpha(x(t))+\beta(x(t))\dot x_2(t)$ such that $(x(t),\tilde v(t))$ solves $\dot x=\tilde f(x)+\tilde g(x)\tilde v$. Therefore, $\Xi$ has corresponding solutions with  any $(Q,v)$-explicitation $\Sigma$  independently of the choice of $Q$, $E^{\dagger}_1$ and $g$. 
\end{proof}
\begin{proof}[Proof of Theorem \ref{Thm:ex and sys}]
	By the assumptions that  ${\rm rank\,} {E(x)}=const.=q$ and \red{${\rm rank\,}{\tilde E(\tilde x)}=const.=\tilde q$} around $x_p$ and {$\tilde x_p$, respectively}, we have that $\Xi$ and {$\tilde \Xi$} are locally ex-equivalent to $$\Xi':\left[ {\begin{matrix}
		{{E_1}(x)}\\
		0
		\end{matrix}} \right]\dot x = \left[ {\begin{matrix}
		{{F_1}(x)}\\
		{{F_2}(x)}
		\end{matrix}} \right] \ \ \ \ {\rm and} \ \ \ \tilde \Xi':\left[ {\begin{matrix}
		{{{\tilde E}_1}\left( {\tilde x} \right)}\\
		0
		\end{matrix}} \right]\dot {\tilde x }= \left[ {\begin{matrix}
		{\tilde{F_1}\left( {\tilde x} \right)}\\
		{\tilde{F_2}\left( {\tilde x} \right)}
		\end{matrix}} \right],$$ respectively, where $E_1(x)$ and ${{{\tilde E}_1}\left( {\tilde x} \right)}$  are full row rank matrices and their ranks are \red{$q$ and $\tilde q$, respectively}. By Definition~\ref{Def:expl}, we have 
	\begin{align}\label{Eq:mtx_expl}
	\begin{array}{c}
	f(x)=E^{\dagger}_1(x)F_1(x), \ \ {\rm Im\,}g(x)=\ker E_1(x), \ \ h(x)=F_2(x),  \\ \tilde f(\tilde x)=\tilde E^{\dagger}_1(\tilde x)\tilde F_1(\tilde x), \ \ {\rm Im\,}\tilde g(\tilde x)=\ker \tilde E_1(\tilde x), \ \ \tilde h(\tilde x)=\tilde F_2(\tilde x).
	\end{array}
	\end{align}
	Note that the explicitation system is defined up to a feedback, an output multiplication and a generalized output injection.  {Any two control systems belonging to  $\mathbf{Expl}(\Xi)$ are sys-equivalent to each other and so are any two control systems belonging to $\mathbf{Expl}(\tilde \Xi)$}. Thus {the} choice of {an} explicitation system makes no difference for the proof of  sys-equivalence. Without loss of generality, we will use $f(x)$, $g(x)$, $h(x)$ and $\tilde f(x)$, $\tilde g(x)$, $\tilde h(x)$ given in (\ref{Eq:mtx_expl}) {for} the remaining part of this proof.
	
	\emph{If.} Suppose $\Sigma\mathop  \sim \limits^{sys}\tilde \Sigma$ { in a neighborhood $U$ of $x_p$}. By Definition \ref{Def:sys-equivalence}, there exists a diffeomorphism $\tilde x=\psi(x)$ and $\beta:U\rightarrow GL(m,\mathbb{R})$ such that $\tilde g\circ\psi = \frac{\partial \psi }{\partial x} g\beta$, \red{which implies 
	$$
	{\ker (\tilde E\circ\psi)={\rm span}\{\tilde g_1,...,\tilde g_m\}\circ\psi={\rm span}\left\{{ \frac{\partial \psi }{\partial x} }g_1,...,{ \frac{\partial \psi }{\partial x} }g_m\right\}={ \frac{\partial \psi}{\partial x} }\ker E}
	$$
and $q=\tilde q$ (since $\dim \ker \tilde E=\tilde m=m=\dim \tilde E$)}.	We can deduce from the above equation that there exists $Q_1:U\rightarrow GL(q,\mathbb{R}) $ such that
	\begin{align}\label{Eq:proof1}
	\tilde E_1\circ\psi =Q_1E_1\left( \frac{\partial \psi }{\partial x} \right)^{-1}.
	\end{align}
	Subsequently,  by $\tilde f\circ\psi = \frac{\partial \psi }{\partial x}(f+\gamma h+g\alpha)$ of Definition \ref{Def:sys-equivalence}, we have
	\begin{align*}
	{(\tilde E^{\dagger}_1\circ \psi)(\tilde F_1\circ \psi) =  \frac{\partial \psi }{\partial x} (E^{\dagger}_1F_1+\gamma F_2+g\alpha). }
	\end{align*}
	Pre-multiply the above equation by $\tilde E_1\circ\psi =Q_1E_1\left( \frac{\partial \psi }{\partial x} \right)^{-1}$, to {obtain}
	\begin{align}\label{Eq:proof2}
	\tilde F_1\circ\psi  =Q_1F_1+Q_1E_1\gamma F_2.
	\end{align}
	Then by $\tilde h\circ\psi =\eta h$ of Definition \ref{Def:sys-equivalence}, we immediately get
	\begin{align}\label{Eq:proof3}
	\tilde F_2\circ\psi =\eta F_2.
	\end{align}
	Now  combining (\ref{Eq:proof1}), (\ref{Eq:proof2}) and (\ref{Eq:proof3}), {we conclude} that $\Xi'$ and $\tilde \Xi'$ are ex-equivalent via $\tilde x=\psi(x)$ and $Q=\left[ {\begin{smallmatrix}
		Q_1&Q_1E_1\gamma \\
		0&\eta 
		\end{smallmatrix}} \right]$, which implies that $\Xi\mathop  \sim \limits^{ex}\tilde \Xi$ (since $\Xi\mathop  \sim \limits^{ex}\Xi'$ and $\tilde \Xi\mathop  \sim \limits^{ex}\tilde \Xi'$).
	
	\emph{Only if.} Suppose that locally $\Xi\mathop  \sim \limits^{ex}\tilde \Xi$ around $x_p$. It follows that  locally $\Xi'\mathop  \sim \limits^{ex}\tilde \Xi'$ around $x_p$, which implies that $q=\tilde q$.  Assume that they are ex-equivalent {via $Q: U\rightarrow GL(l,\mathbb{R})$ and $\tilde x =\psi(x)$ {defined on a neighborhood $U$ of $x_p$}. Let $Q=\left[ {\begin{smallmatrix}
			{{Q_1}}&{{Q_2}}\\
			{{Q_3}}&{{Q_4}}
			\end{smallmatrix}} \right]$,} where $Q_1$, $Q_2$, $Q_3$ and $Q_4$ are matrix-valued functions of {sizes} $q\times q$, $q\times m$, $p\times q$ and $p\times p$, respectively. Then by {$\left[ {\begin{smallmatrix}
			{{Q_1}}&{{Q_2}}\\
			{{Q_3}}&{{Q_4}}
			\end{smallmatrix}} \right]\left[ {\begin{smallmatrix}
			{{E_1}}\\
			0
			\end{smallmatrix}} \right] = \left[ {\begin{smallmatrix}
			\tilde E_1\circ\psi\\
			0
			\end{smallmatrix}} \right]\frac{\partial \psi }{\partial x}$,} we can deduce that $Q_3=0$ and $Q_1$, $Q_4$ are invertible matrices. Then we have
	$$
	\left[ {\begin{matrix}
		{{Q_1}}&{{Q_2}}\\
		0&{{Q_4}}
		\end{matrix}} \right]\left[ {\begin{matrix}
		{{E_1}}\\
		0
		\end{matrix}} \right] = \left[ {\begin{matrix}
		\tilde E_1\circ \psi\\
		0
		\end{matrix}} \right]\frac{\partial \psi }{\partial x}, \ \ \ \ \ \  \left[ {\begin{matrix}
		{{Q_1}}&{{Q_2}}\\
		0&{{Q_4}}
		\end{matrix}} \right]\left[ {\begin{matrix}
		F_1\\
		F_2
		\end{matrix}} \right] = \left[ {\begin{matrix}
		{\tilde{F_1}\circ \psi}\\
		{\tilde{F_2}\circ \psi}
		\end{matrix}} \right],
	$$ 
	which implies
	\begin{align}\label{Eq:relation1}
	{\tilde E_1\circ \psi=Q_1E_1\left( \frac{\partial \psi }{\partial x} \right)^{-1},\ \ \tilde F_1\circ \psi =Q_1F_1+Q_2F_2, \ \ \tilde F_2\circ \psi =Q_4F_2.}
	\end{align}
	Thus by  ${\rm Im\,}g(x)=\ker E(x)=\ker E_1(x)$ and  ${\rm Im\,}\tilde g(x)=\ker \tilde E(\tilde x)=\ker \tilde E_1(\tilde x)$, and {using} (\ref{Eq:relation1}), we have 
	\begin{align}\label{Eq:proof4}
	\tilde g\circ \psi =\frac{\partial \psi }{\partial x} g\beta
	\end{align}
	for some $\beta: U\rightarrow GL(m,\mathbb{R})$. Moreover, there exists $\alpha:U\rightarrow \mathbb R^m$ such that
	\begin{align}\label{Eq:proof5}
	\tilde f\circ \psi&=\tilde E^{\dagger}_1\circ \psi\tilde F_1\circ \psi\mathop=\limits^{(\ref{Eq:relation1})} \frac{\partial \psi }{\partial x} E^{\dagger}_1Q^{-1}_1 Q_1F_1+Q_2F_2= \frac{\partial \psi }{\partial x} E^{\dagger}_1Q^{-1}_1\left(Q_1F_1+Q_2F_2+Q_1E_1g\alpha\right)\nonumber\\&=  \frac{\partial \psi }{\partial x} \left(f+E^{\dagger}_1Q^{-1}_1Q_2y+g\alpha\right).
	\end{align}
	In addition, we have 
	\begin{align}\label{Eq:proof6}
	\tilde h\circ \psi  =\tilde F_2\circ \psi\mathop=\limits^{(\ref{Eq:relation1})}Q_4F_2=Q_4h.
	\end{align}
	Finally, it can be seen from (\ref{Eq:proof4}), (\ref{Eq:proof5}), and (\ref{Eq:proof6}) that $\Sigma\mathop  \sim \limits^{sys}\tilde \Sigma$ via $\tilde x=\psi(x)$, $\alpha$, $\beta$, $\gamma=E^{\dagger}_1Q^{-1}_1Q_2$ and $\eta=Q_4$.
\end{proof}
\begin{proof}[Proof of Proposition \ref{Pro:output zeroing vs invariant}]
	We first show that the {sequence of submanifolds} $M^c_k$ of the geometric \red{reduction} method  of the DAE $\Xi$ and {the sequence} $N^c_k$ of the zero dynamics {algorithm} of {any} control system $\Sigma=(f,g,h)\in \mathbf{Expl}(\Xi)$ locally coincide.   Suppose that ${\rm rank\,} E(x)=const.=q$ in a neighborhood $U_1$ of $x_p$. Then there always exists an invertible matrix $Q(x)$ defined on $U_1$ such that  $E_1(x)$ of
	$Q(x)E(x)=\left[ {\begin{smallmatrix}
		{{E_1}(x)}\\
		0
		\end{smallmatrix}} \right]$ is of full row rank $q$ for all $x\in U_1$, denote $Q(x)F(x)=\left[  {\begin{smallmatrix}
		F_1(x)\\
		F_2(x)
		\end{smallmatrix}} \right] $.
	Recall, \red{see} Remark \ref{rem:zerodynamic}, that $N_k$ of the zero dynamics algorithm  \red{are} well-defined for \cyan{any $\Sigma\in \mathbf{Expl}(\Xi)$}   and that $N_k$ are the same for all control systems belonging to $ \mathbf{Expl}(\Xi)$.
	So the choice of  \red{an} explicitation system makes no difference for $N_k$. We may choose a control system $ \Sigma=(f,g,h)\in\mathbf{Expl}(\Xi)$,  given by $f(x)=E_1^{\dagger }(x)F_1(x)$,  ${\rm Im\,} g(x)=\ker {E}(x)$, $h(x)=F_2(x)$. By the definition of $M_1$ (see (\ref{Eq:Mseq})) and $N_1=h^{-1}(0)$,  we have
	\begin{align*}
	M^c_1={M_1}\cap U_1 &= \left\{ {x \in U_1:Q(x)F(x) \in {\rm Im\,}Q(x)E(x)} \right\}=\left\{ x \in U_1:\left( {\begin{smallmatrix}
		F_1(x)\\
		F_2(x)
		\end{smallmatrix}} \right) \in {\rm Im\,}\left[ {\begin{smallmatrix}
		{{E_1}(x)}\\
		0
		\end{smallmatrix}} \right] \right\}\\&=\left\{ x \in U_1: F_2(x)=0 \right\}=\left\{ x \in U_1: h(x)=0 \right\}=N_1\cap U_1=N^c_1.
	\end{align*}
	For $k>1$, suppose $M^c_{k-1}=N^c_{k-1}$. Then by (\ref{Eq:Mseq}) and (\ref{Eq:Nseq}), we have
	\begin{align*}
	&{M_k} = \left\{ x \in M^c_{k-1}:Q(x)F(x) \in Q(x)E(x)T_xM^c_{k-1} \right\}=\left\{ x \in  M^c_{k-1}:\left( {\begin{smallmatrix}
		F_1(x)\\
		F_2(x)
		\end{smallmatrix}} \right) \in \left[ {\begin{smallmatrix}
		{{E_1}(x)}\\
		0
		\end{smallmatrix}} \right]T_xM^c_{k-1} \right\}\\
	&=\left\{ x \in  M^c_{k-1}:
	F_1(x) \in {{E_1}(x)}T_xM^c_{k-1} \right\}=\left\{ x \in M^c_{k-1}:f(x)+\ker E_1(x)\subseteq T_xM^c_{k-1}+\ker E_1(x)\right\}\\
	&=\left\{ {x \in N^c_{k-1}:f(x) \in {T_x N^c_{k-1}}}+\mathcal G(x)\right\}=N_k,
	\end{align*}
\red{and thus $M^c_k=N^c_k$.}	If either one {among} (A1) and (A2) \red{is satisfied}, {then} by  $N^c_k=M^c_k$, we can easily deduce the other one and thus  (A1) and (A2) \cyan{are equivalent}. {Then} by Proposition \ref{Pro:invariant manifold},  $M^*=M^c_{k^*}$ is a locally maximal invariant submanifold and by Proposition 6.1.1 of \cite{Isid95}, $N^*=N^c_{k^*}$ is a local maximal output zeroing submanifold. Moreover, we have locally $M^*=N^*$ (since locally $M^c_k=N^c_k$). 	Now  under the assumption that $\dim\, E(x)T_xM^*=const.$ for all $x\in M^*$ around $x_p$, by Theorem~\ref{Thm:1}, $\Xi$ is internally regular if and only if $\dim\, M^*=\dim\, E(x)T_xM^*$, i.e., $\ker E(x)\cap T_xM^*=0$, locally $\forall x\in M^*$ around $x_p$.  Thus by $N^*=M^*$ and $\ker E(x)=\red{\mathcal G(x)}$,  we have that $\Xi$ is internally regular (around $x_p$) if and only if $\red{\mathcal G(x_p)}\cap T_{x_p}N^*=0$. 	
\end{proof}
\subsection{Proof of Theorem \ref{Thm:EXWI}}\label{sec:prf_Thm3}
\begin{proof}
	$(i)\Rightarrow(ii)$: Suppose in a neighborhood $U$ of $x_p$ that ${\rm rank\,}E(x)=q$ and $\mathcal G(x)=\ker E(x)={\rm span} \{g_1(x),\ldots,g_{m}(x)\}$ is involutive,  where $g_1,\ldots,g_m$ are independent vector fields on $U$ and $m=n-q$. Then by the involutivity of $\mathcal G$, there exist local {coordinates} $\tilde x=(\tilde x_1, \tilde x_2)=\psi(x)$, {where $\tilde x_1=(\tilde x^1_1,\ldots,\tilde x^q_1)$ and $\tilde x_2=(\tilde x^1_2,\ldots,\tilde x^{n-q}_2)$,} such that {${\rm span}\left\lbrace {\rm d} \tilde x_1^1,\ldots,{\rm d} \tilde x^{q}_1 \right\rbrace={\rm span}\left\lbrace {\rm d}\tilde  x_1\right\rbrace  =\mathcal G^{\bot}$} (Frobenius theorem \cite{kobayashi1963foundations}), where $\mathcal G^{\bot}$ denotes the co-distribution which annihilates $\mathcal G$.  Note that in {the} $\tilde x$-coordinates, the distribution
	\begin{align*}
	\ker \tilde E(\tilde x)&=\ker \left( E(x)\left( \frac{\partial \psi(x)}{\partial x}\right) ^{-1}\right) =\frac{\partial \psi (x)}{\partial x}\mathcal G(x)  ={\rm span} \{\tilde g_1(\tilde x),...,\tilde g_{m}(\tilde x)\},
	\end{align*}
	where $\tilde g_i\circ \psi=\frac{\partial \psi }{\partial x}g_i$, $ i=1,...,m $.
	Now let $\tilde g$ be a matrix whose columns consist of $\tilde g_i$, for $i=1,...,m$. It follows that ${\rm rank\,} \tilde g(\tilde x)=m$ around $\tilde x_0=\psi(x_0)$. By $d \tilde x_1=\mathcal G^{\bot}$, we have $\left\langle {\rm d}\tilde x_1,\tilde g_i\right\rangle =0$, {for} $i=1,\ldots,m$. Thus $\tilde g(\tilde x)$ is of the form $\tilde g(\tilde x)=\left[ {\begin{smallmatrix}
		0\\
		{{\tilde g_2}\left( \tilde x \right)}
		\end{smallmatrix}} \right]$,  where {$\tilde g_2:\psi(U)\rightarrow\mathbb{R}^{{m}\times {m}}$}. Since  ${\rm rank\,}\tilde g(\tilde x)=m$,  it can be seen that $\tilde g_2(\tilde x)$ is an invertible matrix, which implies by   $\im \tilde g(\tilde x)=\ker \tilde E(\tilde x)$  that $\tilde {E}(\tilde x)$ has to be of the form  $\tilde {E}(\tilde x)=\left[ {\begin{smallmatrix}
		{{{\tilde E}_1}\left( {\tilde x} \right)}&0
		\end{smallmatrix}} \right]$, where ${\tilde E}_1:\psi(U)\rightarrow\mathbb{R}^{l\times {m}}$. Thus in the {$\tilde x$-coordinates}, $\tilde \Xi=(\tilde E,\tilde F)$ admits the following form:
	\begin{align*}
	\left[ {\begin{matrix}
		{{{\tilde E}_1}\left( {\tilde x} \right)}&0
		\end{matrix}} \right]\left[ {\begin{matrix}
		{{\dot{ \tilde x}_1}}\\
		{{\dot{ \tilde x}_2}}
		\end{matrix}} \right] = \tilde F\left( {\tilde x} \right).
	\end{align*}
where $\tilde F\circ \psi=F$.	Now by ${\rm rank\,}{ E(x)}=q$, we get ${\rm rank\,}\left[ {\begin{matrix}
		{{{\tilde E}_1}\left( {\tilde x} \right)}&0
		\end{matrix}} \right]={\rm rank\,}{ E(x)}=q$ (the coordinate transformation preserves the rank). Thus  there exists $Q: \psi(U)\rightarrow GL(l,\mathbb{R})$ such that $Q(\tilde x)\tilde E(\tilde x)=Q(\tilde x)\left[ {\begin{matrix}
		{{{\tilde E}_1}\left( {\tilde x} \right)}&0
		\end{matrix}} \right]=\left[ {\begin{smallmatrix}
		{\tilde E_1^1\left( {\tilde x} \right)}&0\\
		0&0
		\end{smallmatrix}} \right]$, where $	\tilde E_1^1:\psi(U)\rightarrow\mathbb{R}^{q\times q}$. Since $Q(\tilde x)$ preserves the rank of $\tilde E(\tilde x)$, we have ${\rm rank\,}{\tilde E_1^1\left( {\tilde x} \right)}=q$. Therefore, ${\tilde E_1^1\left( {\tilde x} \right)}$ is an invertible matrix. Now let $Q'(\tilde x) = \left[ {\begin{smallmatrix}
		{{{\left( {\tilde E_1^1\left( {\tilde x} \right)} \right)}^{ - 1}}}&0\\
		0&{{I_{m}}}
		\end{smallmatrix}} \right]Q(\tilde x)$ and denote $Q'(\tilde x)\tilde F(\tilde x) = \left[ {\begin{smallmatrix}
		{{F_1}\left( {\tilde x} \right)}\\
		{{F_2}\left( {\tilde x} \right)}
	\end{smallmatrix}} \right]$. It is seen that, via $\tilde x =\psi(x)$ and $Q'$, $\Xi$ is locally ex-equivalent to $\tilde \Xi=(Q'\tilde E, Q'\tilde F)$, where $Q'\tilde E\circ\psi=Q'E(\frac{\partial \psi }{\partial x})^{-1}=\left[ {\begin{smallmatrix}
		{{I_q}}&0\\
		0&0
		\end{smallmatrix}} \right]$. {Clearly}, $\tilde \Xi$ is a semi-explicit DAE.
	
	$(ii)\Rightarrow (iii)$: Suppose that $\Xi$ is locally  ex-equivalent to $\Xi^{SE}$ of \cyan{the form}  (\ref{Eq:DAE1}) around $x_p$. Then, any control system $\Sigma\in\textbf{Expl}(\Xi) $ is locally sys-equivalent to $\Sigma'\in\textbf{Expl}(\Xi^{SE})$ below (by Theorem~\ref{Thm:1}):
	$$
	\Sigma':\left\lbrace  {\begin{array}{c@{\ }l}
		\left[ {\begin{smallmatrix}
			{{{\dot {x}}_1}}\\
			{{{\dot { x}}_2}}
			\end{smallmatrix}} \right] &= \left[ {\begin{smallmatrix}
			{{F_1}(x_1,x_2)}\\
			0
			\end{smallmatrix}} \right] + \left[ {\begin{smallmatrix}
			0\\
			{{I_{m}}}
			\end{smallmatrix}} \right]v,\\
		y &=  F_2(x_1,x_2).
		\end{array}} \right.
	$$
	{Suppose that  $\Sigma\mathop  \sim \limits^{sys} \Sigma'$ via $z=(z_1,z_2)=\psi(x)$, $\alpha$, $\beta$ and $\gamma=\left[ {\begin{smallmatrix}
			\gamma_1\\
			\gamma_2
			\end{smallmatrix}} \right]$}, then
	$$
	\Sigma:\left\lbrace  {\begin{aligned} 
		\left[ {\begin{smallmatrix}
			{{{\dot {z}}_1}}\\
			{{{\dot { z}}_2}}
			\end{smallmatrix}} \right] &= \frac{\partial \psi(x)}{\partial x}\left( \left[ {\begin{smallmatrix}
			{{F_1}(x)}\\
			0
			\end{smallmatrix}} \right] + \left[ {\begin{smallmatrix}
			\gamma_1(x)\\
			\gamma_2(x)
			\end{smallmatrix}} \right]y+\left[ {\begin{smallmatrix}
			0\\
			{{I_{m}}}
			\end{smallmatrix}} \right](\alpha(x)+\beta(x)\tilde v) \right), \\
		\tilde y &=  \eta(x) F_2(x).
		\end{aligned}}  \right.
	$$
	{By Definition \ref{Def:reduce}, $\Sigma$ can  always be fully reduced to (by a coordinates change and a feedback transformation)
		$$
		\left\lbrace  {\begin{array}{c@{\,}l}
			\dot x_1
			&=  F_1(x_1,x_2) + \gamma_1(x_1,x_2)F_2(x_1,x_2),\\
			y &=  \eta(x_1,x_2) F_2(x_1,x_2),
			\end{array}}  \right.
		$$
		where $x_2$ is the new control.}

	$(iii)\Rightarrow(i)$: Suppose $(iii)$ {holds}. Then  $\mathbf{Expl}(\Xi)$ is not empty {implies} that   $E(x)$ has constant rank around $x_p$. By Definition \ref{Def:reduce}, any control system ${\Sigma\in \mathbf{Expl}(\Xi)}$ can be fully reduced implies $\red{\mathcal G}=\ker E(x)={\rm span}\left\lbrace g_1,...,g_{m}\right\rbrace $ is involutive.
\end{proof}
\subsection{Proof of Theorem \ref{Thm:NWF}}\label{sec:prf_Thm4}
\begin{claim}\label{Cl:claim1}
	If assumptions (A1)-(A3) of Theorem \ref{Thm:NWF} are satisfied, then {the point} $x_p$ is a regular point of the zero dynamics algorithm (rank conditions (i), (ii), (iii) of Proposition 6.1.3 of \cite{Isid95} are satisfied) for any control system $\Sigma\in \mathbf{Expl}(\Xi)$. If so, we use Proposition 6.1.5 of \cite{Isid95} with a small modification: there exist local coordinates $(z,z^*)=(z_1,\ldots,z_m,z^*)$, where $z_i=(z^1_i,\ldots,z^{\rho_i}_i)$, such that $\Sigma$ is the following form
\end{claim} 
\vspace{-0.7cm}
\begin{footnotesize}
	\begin{align}\label{Eq:ZD form}
	\left\lbrace \begin{array}{cl}
	y_1&= z _1^1 \\
	\dot z _1^1 &= z _1^2 + \sigma _1^1v\\
	&\cdots \\
	\dot z _1^{{\rho_1} - 1} &= z _1^{{\rho _1}} + \sigma _1^{{\rho_1} - 1}v\\
	\dot z _1^{{\rho_1}} &= {\alpha_1} + {\beta_1}v\\
		y_2&= z _2^1\\
	\dot z _2^1 &= z _2^2 + \delta_{2,1}^{1}\left( {{\alpha_1} + {\beta_1}v} \right) + \sigma _2^1v\\
	&\cdots \\
	\dot z _2^{{\rho_2} - 1} &= z _1^{{\rho_2}} + \delta_{2,1}^{\rho_2-1}\left( {{\alpha_1} + {\beta_1}v} \right) + \sigma _2^{{\rho_2} - 1}v\\
	\dot z _1^{{\rho_2}} &= {\alpha_2} + {\beta_2}v
	\end{array} \ \ \ \begin{array}{ll} 
	&\vdots\\
		y_i&= z _i^1, \ \ \ i=3,\ldots,m\\
	\dot z _i^1 &= z _i^2 + \sum\limits_{s = 1}^{i - 1} {\delta_{i,s}^{1}\left( {{\alpha_s} + {\beta_s}v} \right)}  + \sigma _i^1v\\
	&\cdots \\
	\dot z _i^{{\rho_i} - 1}& =z _i^{{\rho_i}} + \sum\limits_{s = 1}^{i - 1} {\delta_{i,s}^{\rho_i-1}\left( {{\alpha_s} + {\beta_s}v} \right)}  + \sigma _i^{{\rho_i} - 1}v\\
	\dot z _i^{{\rho_i}} &= {\alpha_i} + {\beta_i}v\\
	&\vdots \\
	\dot z^*&= f^*(z,z^*)+ g^*(z,z^*)v.
	\end{array} \right. 
	\end{align} 
		\end{footnotesize}
	where 	\blue{$\delta^j_{i,s}\equiv0$ for $1\le j<\rho_{s}$, $1\le s\le i-1$}.
\begin{rem}
	(i)	Note that in  (\ref{Eq:ZD form}),  $\rho_1\le \rho_2\le\ldots\le \rho_m$ and the matrix $\beta= (\beta_1,\ldots,\beta_m)$ is invertible at $x_p$. The functions $\sigma_k$ satisfy {$\sigma^{k}|_{N_{k}}=0$ for $k=1,\ldots,\rho_i-1$, where $$N_{k}=\{(z,z^*):z_i^j=0, \ 1\le i\le m,\ 1\le j\le k\}.$$}
	(ii) \blue{There are two differences between {system} (\ref{Eq:ZD form}) and the zero dynamics form of Proposition 6.1.3 of \cite{Isid95}, {where} the functions $\sigma^1_1,\ldots, \sigma_1^{\rho_1-1}$ are not present and all the functions $\delta^j_{i,s}$ can be nonzero. However, in (\ref{Eq:ZD form}), $\sigma^1_1,\ldots, \sigma_1^{\rho_1-1}$ vanish on $N_1,\ldots, N_{\rho_1 -1}$, respectively, but may not outside, and   $\delta^j_{i,s}\equiv0$ for $1\le j<\rho_{s}$, $1\le s\le i-1$.}
\end{rem}

\begin{proof}[Proof of Claim \ref{Cl:claim1}]
	We will prove that assumptions (A1), (A2), (A3) of Theorem \ref{Thm:NWF} correspond to the rank conditions (i), (ii), (iii) of Proposition 6.1.3 in \cite{Isid95}.  By the assumption of Theorem~\ref{Thm:NWF} that  $\rk E(x)=const.$ around $x_p$, we have $\textbf{Expl}(\Xi)$ {is not empty}. Now, in order to compare the two algorithms (the geometric reduction algorithm of Appendix for $\Xi$ and the zero dynamics algorithm  in \cite{Isid95}  for  $\Sigma \in \mathbf{Expl}(\Xi)$), we use the same notations as {in} the  algorithm of Appendix. 
	
	Then for a control system $\Sigma=(f,g,h)\in \mathbf{Expl}(\Xi)$, we have $f(x)=(\tilde E_1^1)^{\dagger}\tilde F_1^1(x)$, ${\rm Im\,}g(x)=\ker E(x)=\ker \tilde \delta^1_1(x)$, $h(x)=\tilde F_1^2(x)$. 
	The zero dynamics algorithm for $\Sigma$ can be implemented in the following way:
	
	Step 1: by (A1) of Theorem \ref{Thm:NWF}, we get ${\rm D}h(x)= {\rm D}\tilde F_1^2(x)$ has constant rank $n-n_1$ around $x_p$ (condition (i) of Proposition 6.1.3 in \cite{Isid95}). Thus $h^{-1}(0)$ can be locally expressed as $N^c_1=\{x:H_1(x)=0\}$, where $H_1=\psi_1(x)=(\psi_1^1,...,\psi_1^{n-n_1})$. 
	
	Step $k$ ($k>1$): By the proof of Proposition \ref{Pro:output zeroing vs invariant}, we have $N^c_{k-1}=M^c_{k-1}$, which is
	\begin{align*}
	N^c_{k-1}=M^c_{k-1}=\{x:H_{k-1}(x)=0\},	
	\end{align*}
	where $H_{k-1}= (\psi_0,\ldots,\psi_{k-1})$.  
	By the zero dynamic algorithms, $N_k$ \cyan{consists of} all $x\in N^c_{k-1}$ such that 
	$$L_fH_{k-1}(x)+L_gH_{k-1}(x)u=0.$$
	Then by {assumption (A2) of} Theorem \ref{Thm:NWF}, we can deduce that \begin{align}\label{Eq:pfconstrank1}
	\dim\,	(\ker E\cap \ker dH_{k-1})(x)=\dim\,({\rm span}\{g_1,\ldots,g_m\}\cap \ker dH_{k-1})(x)=const.,
	\end{align}
	{for all $x\in M^c_{k-1}$ around $x_p$}. Now by $\dim\,\ker E(x)=const.$ around $x_p$ (implied by $\rk E(x)=const.$), we get 
	\begin{align}\label{Eq:pfconstrank2}
	\dim\,{\rm span}\{g_1,\ldots,g_m\}(x)=const.
	\end{align}
	locally around $x_p$. By (\ref{Eq:pfconstrank1}) and (\ref{Eq:pfconstrank2}), we get ${\rm rank\,}L_gH_{k-1}(x)=const.$ for all $x\in M^c_{k-1}$ around $x_p$ (condition (ii) of Proposition 6.1.3 in \cite{Isid95}).

	Since  $\rk L_gH_{k-1}(x)=const.$, there exists a basis matrix $R_{k-1}(x)$ of the annihilator of the image of $L_gH_{k-1}(x)$, that is $R_{k-1}(x)L_gH_{k-1}(x)=0$. Thus $N^c_k$ can be defined by \[N^c_k=\{x\in U_k:H_{k-1}(x)=0, \  R_{k-1}(x)L_fH_{k-1}(x)=0\}.\] Notice that by the geometric reduction algorithm, we have \[M^c_{k}=\{x\in U_k:H_{k-1}(x)=0, \ \ \ \tilde  F^2_k(x)=0\}.\] By $N^c_{k}=M^c_{k}$ and the fact that ranks of the differential  of $(H_{k-1}(x),\tilde F^2_k(x))$ are constant for all $x$ around $x_p$ (assumption (A1) of Theorem \ref{Thm:NWF}),  it follows that the rank of the differential  of $\left[ {\begin{smallmatrix}
		{{H_{k - 1}}(x)}\\
		{{R_{k - 1}}(x){L_f}{H_{k - 1}}(x)}
		\end{smallmatrix}} \right]$ is constant around $x_p$ (condition (i) of Proposition 6.1.3 in \cite{Isid95}).
	
Assumption (A3) of Theorem \ref{Thm:NWF} that $\dim\, E(x)T_{x}M^*=\dim\, M^*$ locally around $x_p$ implies 
	$${\rm span}\left\lbrace  g_1(x_p),\ldots,g_m(x_p)\right\rbrace \cap T_{x_p}N^*=0.$$ Finally, by $N^*=\{x: H_{k^*}(x)=0\}$,
	it follows that the matrix ${L_g}H_{k^*}  (x_p)  $ has rank $m$ (condition (iii) of Proposition 6.1.3 in \cite{Isid95}). 
\end{proof}
\begin{proof} [Proof of Theorem \ref{Thm:NWF}]
	Observe that by assumption (A3) and Theorem \ref{Thm:1}(iii), we have that $\Xi$ is internally regular. Then by Claim \ref{Cl:claim1}, we have $x_p$ is a regular point of the zero dynamics algorithm for any control system $\Sigma\in \mathbf{Expl}(\Xi)$.  Thus there exist local coordinates $(z,z^*)$ such that  $\Sigma$ is in the form  (\ref{Eq:ZD form}) around $x_p$.  Notice that the matrix $\beta=(\beta_1,\ldots,\beta_m)$ is invertible at $x_p$ and the functions $\sigma_i^{k}|_{N^c_k}=0$ for $1\le i\le m$, $1\le k\le \rho_i-1$, which implies $\sigma_i^{k}\in \mathbf{I}^k$,  where $\mathbf{I}^k$ is the ideal generated by $z_i^j$, $1\le i\le m$, $1\le j\le k$ in the ring of smooth functions of $z^a_b$ and $z^*_c$. Then for system (\ref{Eq:ZD form}), using the feedback transformation $\tilde v=\alpha+\beta v$, where $\alpha= (\alpha_1,\ldots,\alpha_m)$, we get  
	\begin{align}\label{Eq:prf1}
	   \left\lbrace \begin{array}{c@{ }l}
	y_i&= z _i^1, \ \ \ \ \ i=1,\ldots,m,\\
	\dot z _i^1 &=z _i^2 + \sum\limits_{s = 1}^{i - 1} {\delta_{i,s}^{1}\tilde v_s }  + a^1_i+b^1_i\tilde v,\\
	&\cdots \\
	\dot z _i^{{\rho_i} - 1}& =z _i^{{\rho_i}} + \sum\limits_{s = 1}^{i - 1} {\delta_{i,s}^{{\rho_i} - 1}\tilde v_s}  +  a^{\rho_i-1}_i+b^{\rho_i-1}_i\tilde v,\\
	\dot z _i^{{\rho_i}} &= \tilde v_i,\\ 
	\dot z^*&= \tilde f^*(z,z^*)+ \tilde G^*(z,z^*)\tilde v,
	\end{array}\right.
	\end{align} 
where $\tilde f^*=f^*-\bar g\beta^{-1}\alpha$, $\tilde G^*=g^*\beta^{-1}$, and where $a^{k}_{i}=-\sigma _i^k\beta^{-1}\alpha$, $b^{k}_{i}=\sigma _i^k\beta^{-1}$, for $1\le i\le m$, $1\le k\le\rho_i-1$ and by $\sigma_i^{k}\in \mathbf{I}^k$, we have  $a^k_i,b^k_{i,s}\in \mathbf{I}^k$.

 Recall from (\ref{Eq:ZD form}) that the functions  $\delta^j_{i,s}\equiv0$ for $1\le j<\rho_{s}$, $1\le s\le i-1$. Then if the function $\delta^j_{i,\bar s}\neq  0$,  $j=\rho_{\bar s}+k$, for a certain $1\le \bar s\le i-1$ and a certain $0\le k\le \rho_i-1-\rho_{\bar s}$, we  show that, via suitable changes   of coordinates and   output  multiplications, the nonzero function  $\delta^{k+\rho_{\bar s}}_{i,\bar s}$ can be eliminated. Namely, define the new coordinates (and keep the remaining coordinates unchanged):
 	$$
 	\tilde z^{k+1}_i=z^{k+1}_i- \delta^{\rho_{\bar s}+k}_{i,\bar s}z^1_{\bar s}, ~ ~ ~  \tilde z^{k+2}_i=z^{k+2}_i-\delta^{\rho_{\bar s}+k}_{i,\bar s}z^2_{\bar s}, ~ ~ ~ \ldots, ~ ~ ~ \tilde z^{k+\rho_{\bar s}}_i=z^{k+\rho_{\bar s}}_i-\delta^{\rho_{\bar s}+k}_{i,\bar s}z^{\rho_{\bar s}},
 	$$
 we have (notice that below $\delta^1_{\bar s,s}\equiv0$ for $1\le s \le \bar s-1$)
 $$
 \begin{aligned}
 \dot {\tilde z}^{k+1}_i&=z^{k+2}_i+\sum\limits_{s = 1}^{i - 1} {\delta_{i,s}^{k+1}\tilde v_s}+ a^{k+1}_i+b^{k+1}_i\tilde v-(\delta^{\rho_{\bar s}+k}_{i,\bar s})'z^1_{\bar s}-\delta^{\rho_s+k}_{i,s}(z^2_{\bar s}+a^1_{\bar s}+b^1_{\bar s}\tilde v+\sum\limits_{s = 1}^{\bar s - 1} {\delta_{\bar s,s}^{1}\tilde v_s})\\
 &=(z^{k+2}_i-\delta^{\rho_{\bar s}+k}_{i,\bar s}z^2_{\bar s})+(a^{k+1}_i-(\delta^{\rho_{\bar s}+k}_{i,\bar s})'z^1_{\bar s}-\delta^{\rho_s+k}_{i,s}a^1_{\bar s})+(b^{k+1}_i-\delta^{\rho_s+k}_{i,s}b^1_{\bar s})\tilde v+\sum\limits_{s = 1}^{i - 1} {\delta_{i,s}^{k+1}\tilde v_s}\\
& =\tilde z^{k+2}_i+\tilde a^{k+1}_i+\tilde b^{k+1}_i \tilde v+\sum\limits_{s = 1}^{i - 1} {\delta_{i,s}^{k+1}\tilde v_s},
 \end{aligned}
 $$	
where $(\delta^{\rho_{\bar s}+k}_{i,\bar s})'$ denotes the derivative of $\delta^{\rho_{\bar s}+k}_{i,\bar s}(x(t))$ with respect to $t$, and $\tilde a^{k+1}_i=a^{k+1}_i-(\delta^{\rho_{\bar s}+k}_{i,\bar s})'z^1_{\bar s}-\delta^{\rho_s+k}_{i,s}a^1_{\bar s}$, $\tilde b^{k+1}_i=b^{k+1}_i-\delta^{\rho_s+k}_{i,s}b^1_{\bar s}$, and it is clear that $\tilde a^{k+1}_i,\tilde b^{k+1}_{i,l}\in \mathbf{I}^{k+1}$. Then via similar calculations, we have  
 $$ 
\dot {\tilde z}^{k+j}_i=\tilde z^{k+j+1}_i+\tilde a^{k+j}_i+\tilde b^{k+j}_i \tilde v+\sum\limits_{s = 1}^{i - 1} {\delta_{i,s}^{k+j}\tilde v_s}, \ \ \ 2\le j\le \rho_{\bar s}-1, 
 $$
 for some $\tilde a^{k+j},\tilde b^{k+j}_{i,l}\in \mathbf{I}^{k+j}$. Moreover, we have
 $$
\begin{aligned}
\dot {\tilde z}^{k+\rho_{\bar s}}_i&=z^{k+\rho_{\bar s}+1}_i+\sum\limits_{s = 1}^{i - 1} {\delta_{i,s}^{k+\rho_{\bar s}}\tilde v_s}+ a^{k+\rho_{\bar s}}_i+b^{k+\rho_{\bar s}}_i\tilde v-(\delta^{\rho_{\bar s}+k}_{i,\bar s})'z^{\rho_{\bar s}}_{\bar s} -\delta^{\rho_{\bar s}+k}_{i,{\bar s}}\tilde v_{\bar s}\\
&=z^{k+\rho_{\bar s}+1}_i+(a^{k+\rho_{\bar s}}_i-(\delta^{\rho_{\bar s}+k}_{i,\bar s})'z^{\rho_{\bar s}}_{\bar s})+ b^{k+1}_i \tilde v+\sum\limits_{s = 1}^{i - 1} {\delta_{i,s}^{k+\rho_{\bar s}}\tilde v_s}-\delta^{k+\rho_{\bar s}}_{i,{\bar s}}\tilde v_{\bar s}\\
& =  z^{k+\rho_{\bar s}+1}_i+\tilde a^{k+\rho_{\bar s}}_i+\tilde b^{k+\rho_{\bar s}}_i \tilde v+\sum\limits_{s = 1}^{\bar s - 1} {\delta_{i,s}^{k+\rho_{\bar s}}\tilde v_s}+\sum\limits_{s = \bar s + 1}^{i-1} {\delta_{i,{\bar s}}^{k+\rho_{\bar s}}\tilde v_s},
\end{aligned}
$$	
 where the functions $\tilde a^{k+\rho_{\bar s}},\tilde b^{k+\rho_{\bar s}}_{i,l}\in \mathbf{I}^{k+\rho_{\bar s}}$. Thus in the above formula, the nonzero function $\delta^{k+\rho_{\bar s}}_{i,{\bar s}}$ is eliminated. Note that if $k=0$, then the change of coordinate $\tilde z^{1}_i=z^1_i- \delta^{\rho_{\bar s}}_{i,\bar s}z^1_{\bar s}$ transforms the first equation $y_i=z^1_i$ of (\ref{Eq:prf1}) into $y_i=\tilde z^{1}_i+ \delta^{\rho_{\bar s}}_{i,\bar s}z^1_{\bar s}$. We define a new output $\tilde y_i=y_i-\delta^{\rho_{\bar s}}_{i,\bar s}z^1_{\bar s}=y_i-\delta^{\rho_{\bar s}}_{i,\bar s}y_{\bar s}$ (which is actually an output multiplication \cyan{of the form $\tilde y_i=\eta_iy$}) such that the first equation of (\ref{Eq:prf1}) becomes $\tilde y_i=\tilde z^{1}_i$. 
 
Repeat the above construction to eliminate all nonzero functions   $\delta^j_{i,s}$ for $ j\ge \rho_{s}$, $1\le s\le i-1$. Then  system (\ref{Eq:prf1}) becomes the following control system  
 $$
 \tilde \Sigma:\left\lbrace \begin{array}{c@{ }l}
 \tilde y_i&=  \tilde z _i^1, \ \ \ \ \ i=1,\ldots,m,\\
 \dot  {\tilde z} _i^1 &=\tilde z _i^2 +    \tilde a^1_i+ \tilde b^1_i\tilde v,\\
 &\cdots \\
 \dot  {\tilde z} _i^{{\rho_i} - 1}& = \tilde z _i^{{\rho_i}}    +   \tilde a^{\rho_i-1}_i+ \tilde b^{\rho_i-1}_i\tilde v,\\
 \dot  {\tilde z} _i^{{\rho_i}} &= \tilde v_i,\\ 
 \dot z^*&= \tilde f^*(z,z^*)+ \tilde G^*(z,z^*)\tilde v.
 \end{array}\right.
$$
where  $a^k_i,b^k_{i,s}\in \mathbf{I}^k$ for  $1\le k\le\rho_i-1$. It is clear that $\Sigma\mathop  \sim \limits^{sys}\tilde \Sigma$ (we used coordinates changes, feedback transformations and output multiplications to transform $\Sigma$ into $\tilde \Sigma$). Then consider the last row of every subsystem of $\tilde \Sigma$, which is $\dot z _i^{{\rho_i}} = \tilde v_i$. By deleting this equation in every subsystem and setting $y_i=0$ for $i=1,\ldots,m$, and replacing the vector $\tilde v$ by $\dot z ^{{\rho}}$, we transform $\tilde \Sigma$ into a DAE $\tilde \Xi$ below.  It is  \cyan{straightforward} to see that $\tilde \Sigma\in \mathbf{Expl}({\tilde \Xi})$.
	$$ 
	\tilde \Xi:	\left\{ \begin{array}{ccl}
	\left[ {\begin{matrix}
		0&{}&{}&{}\\
		1& \ddots &{}&{}\\
		{}& \ddots & \ddots &{}\\
		{}&{}&1&0
		\end{matrix}} \right]\left[ {\begin{matrix}
		{\dot {\tilde z} _i^1}\\
		{\dot {\tilde z} _i^2}\\
		\vdots \\
		{\dot {\tilde z} _i^{{\rho_i}}}
		\end{matrix}} \right]  &=& \left[ {\begin{matrix}
		{{\tilde z} _i^1}\\
		{{\tilde z} _i^2}\\
		\vdots \\
		{{\tilde z} _i^{{\rho_i}}}
		\end{matrix}} \right] + \left[ {\begin{matrix}
		0\\
		\tilde a^1_i+\tilde b^1_i\dot {\tilde z}^{\rho}\\
		\vdots \\
		\tilde a^{\rho_i-1}_i+\tilde b^{\rho_i-1}_i\dot {\tilde z}^{\rho}
		\end{matrix}} \right], \ \ i=1,\dots,m,\\
	- \tilde G^*\left( {{\tilde z} ,z^*} \right){{\dot {\tilde z} }^{{\rho}}} + \dot z^* &=& \tilde f^*\left( {{\tilde z} ,z^*} \right).
	\end{array} \right. 
$$
  Finally, by Theorem \ref{Thm:ex and sys} and $\Sigma\mathop  \sim \limits^{sys}\tilde \Sigma$, we have that $\Xi\mathop  \sim \limits^{ex} \tilde\Xi$ and that $\tilde \Xi$ is in the \textbf{NWF} of (\ref{Eq:NF1}).
\end{proof}
\section{Conclusions}\label{section:4}
In this paper,   we first revise the geometric reduction method for the existence of nonlinear DAE solutions, and then we define the notions of internal and external equivalence, their differences are {discussed} by analyzing their relations with solutions.  We show that the internal regularity (existence and uniqueness of solutions) of a DAE is equivalent to {the fact} that the DAE is internally equivalent to an ODE (without free variables) on its maximal invariant submanifold. A procedure named explicitation with driving variables is proposed to connect nonlinear DAEs with nonlinear control systems. We show that the external equivalence for two DAEs {is} the same as the system equivalence for {their explicitation systems}.  Moreover, we show that $\Xi$ is {externally} equivalent to a semi-explicit DAE if and only if the distribution defined by $\ker E(x)$ is {of} constant rank and involutive. If so, the driving variables of  a control system $\Sigma\in\mathbf{Expl}(\Xi)$ can be fully reduced.  Finally, two nonlinear generalizations of the Weierstrass form \textbf{WF} are proposed based on  the  explicitation method and the notions as zero dynamics, relative degree and invariant distributions of nonlinear control theory.
\bigskip

\noindent\textbf{Acknowledgment:}  The first author is currently supported by Vidi-grant 639.032.733. 

\bibliographystyle{model1-num-names}
\bibliography{bibthesis}

\section{Appendix} \label{sec:Appendix}
\makeatletter
\renewcommand{\fnum@algorithm}{\fname@algorithm}
\makeatother

\begin{algorithm}[h]\caption{Geometric reduction algorithm for nonlinear DAEs}  
	\begin{algorithmic}[1]
			\begin{small} 
		\Require{Consider $\Xi_{l,n}=(E,F)$, fix $x_p\in X$ and let $U_{0}\subseteq X$ be  an open  connected  subset  containing $x_p$.   	Set  $z_0=x$, $ E_0(z_0)=E(x)$, $ F_0(z_0)=F(x)$,  $M^c_0=U_0$, $r_0=l$, $n_0=n$, \red{and $\Xi_0=(E_0,F_0)$. Below all sets $U_k$ are open in $X$ and $W_k$ are open in $M^c_{k-1}$.} 	
		\Stepk{}  	Suppose that we have defined  at  Step $k-1$: an open neighborhood $U_{k-1}\subseteq X$ of $x_p$, a smooth embedded connected submanifold $M^c_{k-1}$ of $U_{k-1}$ and a DAE $\Xi_{k-1}=(E_{k-1},F_{k-1})$ given by  smooth matrix-valued maps
		
		\begin{center}
			$
		E_{k-1} : M^c_{k-1}\!\rightarrow\!\mathbb{R}^{r_{k-1}\times n_{k-1} },~~  F_{k-1}:M^c_{k-1}\rightarrow\mathbb{R}^{r_{k-1}}, 
		$
		\end{center}
 whose arguments are denoted $z_{k-1}\in M^c_{k-1}$.} 
		\\ Rename the maps as $\tilde E_{k} =E_{k-1} $, $\tilde F_{k}=  {F}_{k-1} $ and define $\tilde \Xi_k:=(\tilde E_{k},\tilde F_{k})$.
	\Ensure{There exists an open neighborhood $U_{k} \subseteq U_{k-1}\subseteq  X$ of $x_p$ such that   $\rk  \tilde E_{k}(z_{k-1})=const.=r_k$,  $\forall z_{k-1}\in W_{k}=U_{k}\cap M^c_{k-1}$. }
	\\		Find a smooth map $Q_{k}:W_{k}\rightarrow GL(r_{k-1},\mathbb{R})$, such that $\tilde E_{k}^1$   of  	${Q_{k}}{\tilde E_{k}} = \left[ {\begin{smallmatrix}
			{\tilde E_{k}^1 }\\
			0\\
	\end{smallmatrix}} \right]$ is of full row rank and denote
	 	${Q_{k}}{\tilde F_{k}} =\left[  {\begin{smallmatrix}
			{\tilde F_{k}^1}\\
			{\tilde F_{k}^2}
	\end{smallmatrix}} \right] $,  where $\tilde E_{k}^1:W_{k}\!\rightarrow\!\mathbb{R}^{r_{k}\times n_{k-1}}$, $\tilde F_{k}^2:W_{k}\!\rightarrow\!\mathbb{R}^{r_{k-1}-r_{k}}$ (so all the matrices depend on $z_{k-1}$). \\
		Following (\ref{Eq:Mseq}), define $
		M_{k}=\left\{ {z_{k-1} \in W_{k}\,|\, {\tilde F_{k}^2(z_{k-1}) = 0} } \right\}.$       
		\Assum{ $x_p\in M_{k}$ and  $ {\rm rank\,}{\rm D}\tilde F_{k}^2(z_{k-1})=const.=n_{k-1}-n_{k}$  for  $z_{k-1}\in M_{k}\cap U_{k}$,  by taking a smaller $U_k$ (if necessary).}  \\
	 By Assumption 2, $M_k\cap U_k$ is a smooth embedded submanifold and by taking again a smaller $U_k$, we may assume that $M^c_k=M_k\cap U_k$ is connected and choose  new coordinates 	$(z_{k},\bar z_{k})=\psi_{k}(z_{k-1})$ on $W_{k}$, where $ \bar z_{k}=\bar \varphi_{k}(z_{k-1})=(\bar \varphi_{k}^1(z_{k-1}),..., \bar \varphi_k^{n_{k-1}-n_{k}}(z_{k-1}))$, with ${\rm d}\bar \varphi_k^1(z_{k-1}),..., {\rm d}\bar \varphi_k^{n_{k-1}-n_{k}}(z_{k-1})$ being all independent rows of ${\rm D}\tilde F_{k}^2(z_{k-1})$, and $z_{k}=\varphi_{k}(z_{k-1})=(\varphi_{k}^1(z_{k-1}),..., \varphi_k^{n_{k-1}}(z_{k-1}))$ are any complementary coordinates  such that $\psi_{k}=(\varphi_k,\bar\varphi_k)$ is a local diffeomorphism. \\
		Set ${{\hat E}_{k}}= {Q_{k}}{\tilde E_{k}}{\left( {\frac{{\partial \bar \varphi_k}}{{\partial z_{k-1}}}} \right)^{ - 1}}$,  ${\hat F}_{k} = {Q_{k}} \tilde F_{k}$. 
		By Definition \ref{Def:ex-equivalence}, $\tilde \Xi_{k}\mathop  \sim \limits^{ex} \hat \Xi_{k}=(\hat E_{k},\hat F_{k})$ via $Q_{k}$ and $\psi_k$,  where 	
\begin{align}\label{Eq:tildeXiuk}
			\begin{small}
		\hat \Xi_{k}:  	 \left[ {\begin{matrix}
			\hat E_{k}^1(z_{k},\bar z_{k})&\bar E_{k}^1(z_{k},\bar z_{k})\\
			0&0\\ 
			\end{matrix}} \right]\left[  {\begin{matrix}
			\dot z_{k}\\
			\dot {\bar z}_{k}
			\end{matrix}} \right]  =\left[  \begin{matrix}
		\hat  F_{k}^1(z_{k},\bar z_{k})\\ 
		\hat F_{k}^2(z_{k},\bar z_{k}) 
		\end{matrix} \right] 
		\end{small} 
		\end{align}
	with $\hat E^1_{k}:W_{k}\to \mathbb R^{r_{k}\times n_{k}}$, $\hat F^1_k\circ\psi_k=\tilde F^1_k$, $\hat F^2_k\circ\psi_k=\tilde F^2_k$ and   $[   \hat  E_{k}^1\circ\psi_k \ \ \bar E_{k}^1 \circ\psi_k]=\tilde E^1_k{\left( {\frac{{\partial \psi_k}}{{\partial z_{k-1}}}} \right)^{ - 1}}$.\\
	Set $\bar z_{k}=0$   to define the following reduced and restricted DAE on  $
	M^c_{k}= \left\{ z_{k-1}\in W_{k} \,|\, \bar z_{k}=0 \right\}$  \red{by}
	\begin{align*}
	\Xi_k: E_k(z_k)\dot z_k=F_k(z_k),
	\end{align*}
		where $E_k(z_k)=\hat E_{k}^1(z_{k},0)$, $F_k(z_k)=\hat F_{k}^1(z_{k},0)$ are matrix-valued~maps and
		$
		E_{k} : M^c_{k}\!\rightarrow\!\mathbb{R}^{r_{k}\times n_{k} },~~  F_{k}:M^c_{k}\rightarrow\mathbb{R}^{r_{k}}.$
		\Rep{  Step $k$ for $k=1,2,3,\ldots,$ \textbf{until} $n_{k+1}=n_{k}$, set $k^*=k$.}
		\Result{ Set   $n^*=n_{k^*}=n_{k^*+1}$, $r^*=r_{k^*+1}$, $M^*= M^c_{k^*+1}$, $U^*=U_{k^*+1}$, $z^*=z_{k^*+1}=z_{k^*}$ and $\Xi^*=(E^*,F^*)$ with $E^*=  E_{k^*+1}$, $F^*= F_{k^*+1}$. } 
	\end{small}
	\end{algorithmic}
\end{algorithm}
\begin{rem}\label{rem:Alg1}
	(i) The geometric reduction algorithm is a {constructive application} for Proposition~\ref{Pro:invariant manifold}, but with more assumptions. The \textbf{Assumption 1}   is made to produce the full row rank matrices $\tilde E^1_k$   and  the zero-level set $M_{k}=\left\{ {z_{k-1} \in W_{k}\,|\, {\tilde F_{k}^2(z_{k-1}) = 0} } \right\}$.  The \textbf{Assumption 2} assure that $M_k\cap U_k$ is a smooth embedded submanifold and makes it possible to  use the  components of $\tilde F_{k}^2$ with linearly independent differentials as  a part  of  new  local    coordinates.
	
	(ii) The integers $r_k$, $n_k$ of the geometric reduction algorithm, satisfy,  for each $k\ge 1$,
	$$
	\left\lbrace \begin{array}{ll}
	l= r_0\ge r_1\ge...\ge r_k\ge...\ge 0,& \ \ \ n= n_0\ge n_1\ge...\ge n_k\ge...\ge 0,\\ n_{k-1}\ge r_k,&
  r_{k-1}-r_k\ge n_{k-1}-n_k.
	\end{array}	\right. 	
	$$
	
	 
\end{rem}

\end{document}